\documentclass[psamsfonts, reqno]{amsart}
\usepackage{amssymb,amsfonts}
\usepackage[all,arc]{xy}
\usepackage{enumerate}
\usepackage{mathrsfs}
\usepackage{pstricks}
\usepackage{pst-node}

\newtheorem{thm}{Theorem}[section]

\newtheorem{prop}[thm]{Proposition}

\newtheorem{conj}[thm]{Conjecture}

\newtheorem{prob}[thm]{Problem}

\newtheorem{alt}{Theorem}[section]
\theoremstyle{definition}
\newtheorem{defn}[thm]{Definition}

\newtheorem{exmp}[alt]{Example}

\theoremstyle{remark}
\newtheorem{rem}[thm]{Remark}

\newcommand{\sA}{{\mathcal A}}

\newcommand{\sG}{{\mathcal G}}
\newcommand{\sL}{{\mathcal L}}
\newcommand{\sP}{{\mathcal P}}
\newcommand{\bv}{{\bf v}}

\newcommand{\NN}{{\mathbb N}}
\newcommand{\QQ}{{\mathbb Q}}
\newcommand{\ZZ}{{\mathbb Z}}
\numberwithin{equation}{section}

\newcommand{\stdnodesep}{3}
\newcommand{\doffset}{3pt}

\newcommand{\node}[3]{\rput{0}(#2){\ovalnode{#3#1}{\Large #1}}}

\newcommand{\aline}[3]{%
	\ncline[nodesepA=\stdnodesep,nodesepB=\stdnodesep]%
	{->}{#1}{#2}%
	\Aput{#3}%
}

\newcommand{\bline}[3]{%
	\ncline[nodesepA=\stdnodesep,nodesepB=\stdnodesep]%
	{->}{#1}{#2}%
	\Bput{#3}%
}

\newcommand{\dline}[4]{%
	\ncarc[nodesepA=\stdnodesep,nodesepB=\stdnodesep,offset=\doffset]%
	{->}{#1}{#2}%
	\Aput{#3}%
	\ncarc[nodesepA=\stdnodesep,nodesepB=\stdnodesep,offset=\doffset]%
	{->}{#2}{#1}%
	\Aput{#4}%
}

\newcommand{\bcircle}[3]{%
	\nccircle[angleA=#2,nodesepA=\stdnodesep]{->}{#1}{20pt}%
	\Bput{#3}%
}

\title{Intersections of multiplicative translates of \\$3$-adic Cantor sets}

\author{William  Abram}
\author{Jeffrey C. Lagarias}
\thanks{The first author received support from an NSF Graduate Research Fellowship. The second author received support from NSF grants DMS-0801029 and DMS-1101373.}
\address{Department of Mathematics and Computer Science, Hillsdale College, Hillsdale, MI 49242-1205,USA}
\email{wabram@hillsdale.edu}
\address{Department of Mathematics, University of Michigan,
Ann Arbor, MI 48109-1043,USA}
\email{lagarias@umich.edu}

\date{August 13, 2013}

\begin{document}

\begin{abstract}

Motivated by a question of Erd\H{o}s, this paper considers 
 questions concerning the discrete dynamical system on the 
$3$-adic integers $\mathbb{Z}_{3}$ given by multiplication by $2$.
Let the  $3$-adic Cantor set $\Sigma_{3, \bar{2}}$ consist of all $3$-adic integers whose expansions
use only the digits $0$ and $1$.  
The  exceptional set  $\mathcal{E}(\mathbb{Z}_3)$ is the set of all elements of $\mathbb{Z}_3$
 whose forward orbits under  this action  intersects the $3$-adic Cantor set $\Sigma_{3, \bar{2}}$ 
infinitely many times.
It has been shown that this set has Hausdorff dimension at most $\frac{1}{2}$
and it has been  conjectured that it has Hausdorff dimension $0$. 
Approaches to upper bounds on 
the Hausdorff dimensions of these sets leads to study
of intersections of multiplicative translates of Cantor sets by powers of $2$.
More generally, this paper  studies the structure of  finite intersections of general
multiplicative translates 
$S= \Sigma_{3, \bar{2}} \cap \frac{1}{M_1} \Sigma_{3, \bar{2}} \cap \cdots \cap \frac{1}{M_n} \Sigma_{3, \bar{2}}$ by
integers $1 < M_1 < M_2 < \cdots < M_n$.
These sets  are describable as sets of $3$-adic
integers whose $3$-adic expansions have one-sided symbolic dynamics given by a finite automaton.
As a consequence, the Hausdorff dimension of such a set is 
always of the form $\log_{3}(\beta)$ where $\beta$ is an algebraic integer. 
This paper gives a method to determine the automaton for given data $(M_1, ..., M_n)$.
Experimental results indicate that 
the Hausdorff dimension of  such sets depends in a very 
complicated way on the integers $M_1, ..., M_n$.

\end{abstract}

\maketitle

 
\tableofcontents

%
%
%
\section{Introduction}

We study the following problem. 
Let the  $3$-adic Cantor set $\Sigma_3 := \Sigma_{3, \bar{2}}$ be the subset of all
$3$-adic integers whose $3$-adic expansions consist of digits $0$ and $1$ only. 
This set is a well-known fractal having Hausdorff dimension $\dim_{H}(\Sigma_{3}) = \log_3 2 \approx 0.630929$.
By a {\em multiplicative translate} of such a Cantor set we mean 
a multiplicatively rescaled set $r \Sigma_{3}= \{ r x: x\in \Sigma_3\}$,
where we restrict to  $r= \frac{p}{q} \in \QQ^{\times}$ being a rational number  that is $3$-integral, 
meaning that $3$ does not divide $q$, 
In this paper we study  sets
 given as finite intersections of such  multiplicative translates:
\begin{equation}\label{eq100}
 C( r_1, r_2, \cdots, r_n) := \bigcap _{i=1}^{n} \frac{1}{r_i} \Sigma_3,
\end{equation}
where now each $\frac{1}{r_i}$ is $3$-integral. These sets
are fractals and our object is to obtain bounds on their Hausdorff dimensions.
Our motivation for
studying this  problem arose from a problem of Erd\H{o}s \cite{Erd79} which is described in
Section \ref{sec11}.

In principle the Hausdorff dimensions of sets $C( r_1, r_2, \cdots, r_n)$ are explicitly computable
in a closed form.  This comes about as follows. 
We show  each such set has the property that 
the $3$-adic expansions
of all the members of $C( r_1, r_2, \cdots, r_n)$ are characterizable as the output labels of all infinite
paths in a labeled finite automaton which  start from a marked initial vertex.
General sets of such path  labels 
associated to   a finite automaton form  symbolic dynamical systems that we call  {\em path sets}
and which we study in   \cite{AL12a}.  The sets $C( r_1, r_2, \cdots, r_n)$
are then $p$-adic path set fractals (with $p=3$), using  terminology we introduced in \cite{AL12b}.
These sets are  collections of all $p$-adic numbers whose $p$-adic expansions have digits
described by the labels along infinite paths according to  a digit assignment map taking path
labels in the graph to $p$-adic digits.
In \cite[Theorem 2.10]{AL12b} we showed  that a  {\em $p$-adic path set fractal}
is any set $Y$ in $\ZZ_p$ constructed 
 by a $p$-adic analogue of a real number graph-directed fractal
construction, as given in  Mauldin and Williams \cite{MW86}.
This geometric object $Y$ is given as the  set-valued fixed point of a dilation functional equation
using a set of $p$-adic affine maps, cf. \cite[Theorem 2.6]{AL12b}. 
We  showed in  \cite[Theorem 1.4]{AL12b}
 that the if $X$ is  a $p$-adic path set fractal 
then any multiplicative translate $r X$ 
by  a $p$-integral rational number $r$
is also a $p$-adic path set fractal. In addition $p$-adic
path set fractals are closed under set intersection, a property they
inherit from path sets, see \cite[Theorem 1.2]{AL12a}.
 Since the $3$-adic Cantor set is a $3$-adic
path set fractal, the full shift on two symbols, these closure properties immediately
imply that every set $C( r_1, r_2, ..., r_n)$ is a $3$-adic path set fractal. 

In  \cite[Theorem 1.1]{AL12b} we showed that 
the  Hausdorff dimension
of a $p$-adic path set fractal $X$ is directly computable from
the adjacency matrix of a suitable presentation $X$. One has
\[
\dim_{H}(X) = \log_p  \beta,
\] 
in which $\beta$ is the spectral radius $\rho(\mathbf{A})$
 of the adjacency matrix $\mathbf{A}$ of 
 a  finite automaton which gives a suitable presentation of the given
path set; see Section \ref{sec2}. 
This spectral radius coincides with
the   {\em Perron eigenvalue} (\cite[Definition 4.4.2]{LM95}) of the   nonnegative integer matrix $\mathbf{A} \neq 0$,
which is 
the largest real  eigenvalue $\beta \ge 0$ of $\mathbf{A}$.
For adjacency matrices of  graphs containing 
at least one directed cycle, which are nonnegative integer matrices,
 the Perron eigenvalue is necessarily a real algebraic integer, and also has $\beta \ge 1.$
In the case at hand we know a priori that   $1 \le \beta \le 2.$ 
Everything here is algorithmically effective, as discussed in Sections \ref{sec2} and \ref{sec3}.

This paper presents   theoretical and experimental
results about these sets.  In Section \ref{sec3} we give an algorithm to compute an
efficient presentation of the underlying path set of $ \mathcal{C}(1, M)$ for integers $M \ge 1$,
which is simpler than the  general constructions given in \cite{AL12a}, \cite{AL12b}.
We extend this method to $ \mathcal{C}(1, M_1, M_2, ..., M_n)$.
We give a complete
analysis of the structure of the resulting path set presentations for two infinite families $\mathcal{C}(1, M_k)$
of integers $\{ M_k:  k \ge 1\}$ whose $3$-adic
expansions take an especially simple form. 
These examples exhibit rather complicated automata
in the presentations.  We experimentally  use the algorithm for $ \mathcal{C}(1, M_1, M_2, ..., M_n)$
to compute various  examples  indicating that the automata depend in an extremely
 complicated way on $3$-adic  arithmetic properties of $M$. 
 This complexity is reflected in the behavior of the Hausdorff dimension function,
 and leads to many open questions.

%
%
%
\subsection{Motivation: Erd\H{o}s problem}\label{sec11}

Erd\H os \cite{Erd79} conjectured that for every $n \geq 9$, the ternary expansion of $2^n$ contains  the 
ternary digit $2$. 
A weak version of this conjecture asserts that there are finitely many $n$ such that the ternary expansion of $2^n$ 
consists of $0$'s and $1$'s. Both versions of this  conjecture  appear  difficult.

In \cite{Lag09} the second author proposed a $3$-adic generalization of this problem, as follows.
Let $\mathbb{Z}_3$ denote the $3$-adic integers, and let a $3$-adic integer $\alpha$ have 
$3$-adic expansion
\[
(\alpha)_3 := (\cdots a_2 a_1 a_0)_3= a_0 + a_1 \cdot 3 + a_2 \cdot 3^2 + \cdots,       ~~\mbox{with all}~~ a_i \in \{ 0, 1, 2\}.
\]
 
\begin{defn} \label{defn-exceptional}
The {\em $3$-adic exceptional set} $\mathcal{E}(\mathbb{Z}_3)$
is defined by
\[
\mathcal{E}(\mathbb{Z}_3) := 
\{\lambda \in \mathbb{Z}_3 : \text{for infinitely many $n \ge 0$ the expansion $(2^n \lambda)_3$ omits the digit $2$}\}.
\]
\end{defn}
 The weak version  of Erd\H os's conjecture above 
  is equivalent to the assertion that $\mathcal{E}(\mathbb{Z}_3)$
  does not contain the integer $1$.

The exceptional set seems an interesting object in its own right.
It is forward invariant under multiplication by $2$, and   
one may expect it to be a very small set in terms of measure or dimension.
At present  it remains possible 
 that the $\mathcal{E}(\mathbb{Z}_3)$
 is a countable set, or even that it consists of
the single element $\{ 0\}.$ 
In 2009 the  second author put forward the following conjecture 
 asserting that  the exceptional set is small in the sense of Hausdorff dimension ( \cite[Conjecture 1.7]{Lag09}). 

\begin{conj}\label{cj11}{\em (Exceptional Set Conjecture)}
The $3$-adic exceptional set $\mathcal{E}(\mathbb{Z}_3)$ 
has Hausdorff dimension zero, i.e. 
\begin{equation}
\dim_{H}(\mathcal{E}(\mathbb{Z}_3) )=0.
\end{equation}
\end{conj}

\noindent 

As limited evidence in favor  of this conjecture,    the paper  \cite{Lag09} showed  that the 
 Hausdorff dimension of $\mathcal{E} (\mathbb{Z}_3)$ 
 is at most $\frac{1}{2},$ as explained below.
 That  paper  initiated a strategy to obtain upper bounds for
$\dim_{H}(\mathcal{E}(\mathbb{Z}_3) )$  
 based on the containment relation
\begin{equation}\label{1013}
 \mathcal{E}(\mathbb{Z}_3) \subseteq \bigcap_{k=1}^{\infty} \mathcal{E}^{(k)}(\mathbb{Z}_3),
\end{equation} 
where 
\begin{equation}\label{eq120}
\mathcal{E}^{(k)}(\mathbb{Z}_3) := 
\{\lambda \in \mathbb{Z}_3 : \text{at least $k$ values of $(2^n \lambda)_3$ omit the digit 2}\}.
\end{equation}

These sets form a nested family 
 \[
 \Sigma_{3, \bar{2}}=  \mathcal{E}^{(1)}(\mathbb{Z}_3 ) \supseteq \mathcal{E}^{(2)}(\mathbb{Z}_3) \supseteq 
  \mathcal{E}^{(3)}(\mathbb{Z}_3 ) \supseteq \cdots
 \]
The containment relation (\ref{1013}) immediately implies  inequalities relating the Hausdorff dimension of these sets, namely
\begin{equation}\label{1015}
\dim_{H}(\mathcal{E}(\mathbb{Z}_3))\le  \Gamma,
\end{equation}
where $\Gamma$ is defined by
\begin{equation}\label{Gamma}
 \Gamma := \lim_{k \to \infty} \dim_{H}(\mathcal{E}^{(k)}(\mathbb{Z}_3)).
\end{equation}

The inequality \eqref{1015} raises the subsidiary problem of obtaining upper bounds for $\Gamma$,
which in turn requires obtaining  bounds for the individual $\dim_{H}(\mathcal{E}^{(k)}(\mathbb{Z}_3))$. 
We note the possibility that $\dim_{H}(\mathcal{E}(\mathbb{Z}_3))<  \Gamma$ may hold. 

The analysis of the sets $\mathcal{E}^{(k)}(\mathbb{Z}_3)$ for $k \ge 2$ leads to the study of particular
sets of the kind \eqref{eq100}
considered in this paper. 
We have  
\begin{equation}
 \mathcal{E}^{(k)}(\mathbb{Z}_3) = \bigcup_{0 \leq m_1 < \ldots < m_k} \mathcal{C}(2^{m_1},\ldots,2^{m_k}).
\end{equation}

We next give a simplification, showing that for  the purposes of computing
Hausdorff dimension we may, without loss of generality,  restrict this set union to subsets having $m_1=0$
so that $2^{m_1}=1.$

\begin{defn} \label{de15a}
The {\em restricted $3$-adic exceptional set} $\mathcal{E}_1(\mathbb{Z}_3)$ is given by
$$
\mathcal{E}_1(\mathbb{Z}_3) := 
\{\lambda \in \mathbb{Z}_3 : \text{for $n=0$ and infinitely many other $n$,   $(2^n \lambda)_3$ omits the digit $2$}\}.
$$
\end{defn}

It is easy to see that 
$$
\mathcal{E}(\mathbb{Z}_3)= \bigcup_{n=0}^{\infty} \frac{1}{2^n} \mathcal{E}_1(\mathbb{Z}_3).
$$
Since the right side  is a countable union of sets we obtain 
$$
\dim_{H} ( \mathcal{E}(\mathbb{Z}_3)) = \sup_{n \ge 0} \Big(\dim_{H} (\frac{1}{2^n} \mathcal{E}_1(\mathbb{Z}_3))\Big)
= \dim_{H} (\mathcal{E}_1(\mathbb{Z}_3)).
$$
and we also have $\mathcal{E}_1(\mathbb{Z}_3) \subset \Sigma_{3, \bar{2}}$.
Now set 
$$
\mathcal{E}_1^{(k)}(\mathbb{Z}_3) := 
\{\lambda \in \Sigma_{3, \bar{2}} : \text{for at least $k$ values of $n \ge 0$,  $(2^n \lambda)_3$ omits the digit 2}\}.
$$
For $0< m_1<m_2 < \cdots < m_k$ we have  the
set identities
$$\mathcal{C}(2^{m_1},\ldots,2^{m_k}) = \frac{1}{2^{m_1}}\mathcal{C}(1,2^{m_2-m_1},\ldots,2^{m_k-m_1}).$$
These identities    yield 
$
\mathcal{E}^{(k)}(\mathbb{Z}_3) = \bigcup_{n=0}^{\infty} 2^{-n} \mathcal{E}_1^{(k)}(\mathbb{Z}_3).
$
Again, since  this is a countable union of sets, we obtain  the equality
$$
\dim_{H}( \mathcal{E}^{(k)}(\mathbb{Z}_3)) = \sup_{k \ge 1}  \Big( \dim_{H}(2^{-n} \mathcal{E}_1^{(k)}(\mathbb{Z}_3)\Big)= \dim_{H}(\mathcal{E}_1^{(k)}(\mathbb{Z}_3))
$$
asserted above.
It also follows that 
\begin{equation}\label{newgamma}
\Gamma = \lim_{k \to \infty} \dim_{H}(\mathcal{E}_1^{(k)}(\mathbb{Z}_3)).
\end{equation}

We now have
\begin{equation}
\mathcal{E}_1^{(k)}(\mathbb{Z}_3) = \bigcup_{0 \leq m_1 < \ldots < m_{k-1}} \mathcal{C}(1, 2^{m_1},\ldots,2^{m_{k-1}}).
\end{equation}
The right side of this expression  is a countable union of sets, so we have 
\begin{equation}\label{1016}
\dim_{H}( \mathcal{E}_1^{(k)}(\mathbb{Z}_3) ) = 
\sup_{0 \leq m_1 < \ldots < m_{k-1}} \Big( \dim_{H}\big(\mathcal{C}(1, 2^{m_1},\ldots,2^{m_{k-1}})\big)\Big) .
\end{equation}
Upper bounds for the right side of this formula are obtained by bounding above
the Hausdorff dimensions of all the individual sets $\mathcal{C}(1, 2^{m_1},\ldots,2^{m_{k-1}})$,
of the form (1.1). Lower bounds may be obtained
by determining the Hausdorff dimension of specific individual sets $\mathcal{C}(1, 2^{m_1},\ldots,2^{m_{k-1}})$.
By this means the second author \cite[Theorem 1.6 (ii)]{Lag09} obtained the upper bound
\begin{equation}\label{lowdim11}
\Gamma \le \dim_{H} (\mathcal{E}^{(2)}(\mathbb{Z}_3) ) =\dim_{H} (\mathcal{E}_1^{(2)}(\mathbb{Z}_3) )
\le \frac{1}{2},
\end{equation}
and using \eqref{1015} we conclude that 
\begin{equation}\label{record-bound}
\dim_{H} (\mathcal{E}(\ZZ_3)) \le \frac{1}{2}.
\end{equation}

%
%
%
\subsection{Generalized  exceptional set problem}\label{sec11a}

We are interested in obtaining improved upper bounds on $\dim_{H}(\mathcal{E}(\mathbb{Z}_3))$.
To progress further with  the approach above, one needs
a better understanding of the structure of sets  $\mathcal{C}(1, 2^{m_1}, ..., 2^{m_k})$, 
with the hope to obtain uniform  bounds on their Hausdorff dimension. 

One approach to upper bounding the exceptional set is to relax its defining conditions to
allow arbitrary  positive integers  $M$ in 
place of powers of $2$. Since  the $3$-adic Cantor set $\Sigma_{3, \bar{2}}$
is forward invariant under multiplication by $3$, we  will restrict to integers
$M \not\equiv 0\, (\bmod \, 3)$.
 
For  application to the exceptional set 
 $\mathcal{E}(\mathbb{Z}_3)$, the discussion in Section \ref{sec11}
 indicates  that it suffices to consider  
 the restricted
 family of sets $\mathcal{C}(1,M_1,\ldots,M_n)$, i.e. taking $M_0=1$.
 We define a relaxed version of the restricted $3$-adic exceptional set, as follows.

\begin{defn} \label{defn-generalized}
The {\em $3$-adic generalized 
exceptional set} is the set
\[\mathcal{E}_{\star}(\mathbb{Z}_3) := \{\lambda \in \mathbb{Z}_3
: \text{there are infinitely many 
$M  \ge 1$, $M \not\equiv 0 ~(\bmod  \, 3)$, including \text{$M=1$},    }\]
\[ \text{such that the $3$-adic expansion} ~(M \lambda)_3 \text{ omits the digit $2$}\}.\]
\end{defn}

When considering intersective sets $C(1,M_1,\ldots,M_n)$, we can then further restrict to require all
 $M_i \equiv 1 ~(\bmod \, 3)$,
 since any  $M \equiv 2 ~(\bmod \, 3)$ has  $C(1,M)= \{ 0\}.$
 We have  
$\mathcal{E}_1 (\mathbb{Z}_3) \subset \mathcal{E}_{\star} (\mathbb{Z}_3)  \subset \Sigma_{3, \bar{2}}$
and therefore
\begin{equation}
\dim_{H}( \mathcal{E}(\mathbb{Z}_3)) = \dim_{H}( \mathcal{E}_1(\mathbb{Z}_3)) \le \dim_{H}( \mathcal{E}_{\star}(\mathbb{Z}_3)).
\end{equation} 
Thus upper bounds for the Hausdorff dimension of the generalized exceptional set yield upper
bounds for that of the exceptional set.

\begin{prob}\label{pr14} {\em (Generalized Exceptional Set Problem )}
Determine  upper and lower bounds for the Hausdorff dimension of
the  generalized exceptional set $\mathcal{E}_{\star} (\mathbb{Z}_3)$. 
In particular, determine whether   $\dim_{H} (\mathcal{E}_{\star} (\mathbb{Z}_3)) =0$ or 
$\dim_{H} (\mathcal{E}_{\star} (\mathbb{Z}_3)) >0$ holds.
\end{prob}

We next define a family of sets in parallel  to $\mathcal{E}_1^{(k)}(\ZZ_3)$  above.
We define
\[
 \mathcal{E}_\star^{(k)}(\mathbb{Z}_3) := 
 \{ \lambda \in \mathbb{Z}_3 : \text{there exist  \,  $1 = M_1 < M_2< \cdots < M_k$, \,with all $M_i  \equiv \, 1 (\bmod\, 3)$},   
  \]
\[ \quad\quad \text{such that the $3$-adic expansion} ~(M_i \lambda)_3 \text{ omits the digit $2$}\}.
\]
 
Then in  parallel to the case above, we have 
\[
\mathcal{E}_\star^{(k)}(\mathbb{Z}_3) = \bigcup_{{1=M_1 < \ldots < M_{k-1}}\atop{M_i \equiv 1 (\bmod ~3)}} 
\mathcal{C}(1, M_1,\ldots,M_{k-1}).
\]
In consequence we have  the inclusion 
\[
\mathcal{E}_\star(\mathbb{Z}_3) \subseteq \bigcap_{k=1}^{\infty} \mathcal{E}_\star^{(k)}(\mathbb{Z}_3).
\]
This inclusion yields the bound
\begin{equation}\label{eq112}
\dim_{H}( \mathcal{E}_\star(\mathbb{Z}_3)) \le   \Gamma_{*},
\end{equation}
where we define
\begin{equation}\label{Gamma-star}
\Gamma_{\star}:= \lim_{k \to \infty} \dim_{H}( \mathcal{E}_\star^{(k)}(\mathbb{Z}_3)).
\end{equation}
As far as we know it is
possible that $\dim_{H}( \mathcal{E}_\star(\mathbb{Z}_3)) <   \Gamma_{*}$ may occur.

The second author  \cite[Theorem  1.6]{Lag09}  obtained the upper bound
\begin{equation}
\Gamma_{\ast} \le \dim_{H}( \mathcal{E}_{\star}^{(2)}(\mathbb{Z}_3)) \le \frac{1}{2},
\end{equation}
which in fact yielded \eqref{lowdim11}. 

Our  interest in the generalized exceptional set problem  stemmed from the fact
that  if it were true that 
$\dim_{H} (\mathcal{E}_{\star} (\mathbb{Z}_3)) =0$,
 then the  Exceptional Set Conjecture \ref{cj11} would follow.  
 However a  main result  of our investigation  
establishes that this does not hold:  we obtain the  lower bounds  
$$
\Gamma_{\star} \ge \dim_{H}(\mathcal{E}_{\star} (\mathbb{Z}_3))  \ge \frac{1}{2} \log_3 2 \approx 0.315464,
$$  
see Theorem \ref{th110} below.  This inconvenient fact limits the upper bounds
attainable on $\dim_{H}(\mathcal{E}(\mathbb{Z}_3))$ via the relaxed problem.

%
%
%
\subsection{Algorithmic  Results}\label{sec12}

We study the size of intersections of multiplicative translates of
the $3$-adic Cantor  set $\Sigma_3:= \Sigma_{3, \bar{2}}$, 
as measured by Hausdorff dimension.
We study the sets
$$\mathcal{C}(1, M_1,\ldots,M_n):=
 \Sigma_{3, \bar{2}} \cap \frac{1}{M_1}\Sigma_{3, \bar{2}} \cap  \cdots \cap \frac{1}{M_n}\Sigma_{3, \bar{2}}.
$$
where $1< M_1 < \cdots < M_n$ are positive integers.
As remarked above, via results in \cite{AL12a}, \cite{AL12b} these 
 sets have a nice description, with their members  having $p$-adic expansions describable
by finite automata, 
which permits effective computation of their Hausdorff dimension. 
These results are reviewed in Section \ref{sec2}, and the
necessary definitions for  presentations of a path set used
in the following theorem appear there.

\begin{thm} \label{th12} {\rm (Dimension of $C(1, M_1, ..., M_n)$)} 

(1) There is a terminating algorithm that takes as input any finite set of integers \\
$1\leq M_1 < \ldots < M_n$, 
and gives  as output a labeled directed graph 
$\mathcal{G}=(G,\mathcal{L})$ with a marked starting vertex $v_0$,   
  which is a presentation of a path set  $X= X(1, M_1, M_2, \cdots , M_n)$ 
  describing the $3$-adic expansions of the  elements of the space
 \[
 \mathcal{C}(1, M_1, ..., M_n):= \Sigma_3 \cap \frac{1}{M_1}\Sigma_3 \cap \ldots \cap \frac{1}{M_n}\Sigma_3.
 \]
 This presentation is right-resolving and all vertices are reachable from the marked vertex. 
 The graph $G$ has at most $\prod_{i=1}^n (1 + \lfloor \frac{1}{2}  M_i \rfloor)$ vertices.

(2) The topological entropy $\beta$ of the path set  $X$ is the Perron eigenvalue of the adjacency matrix $A$ of
the directed graph  $G$.
It is   a real algebraic  integer satisfying $1 \leq \beta \leq 2$. Furthermore the Hausdorff dimension 
$$
\dim_{H} (\mathcal{C}(1, M_1, ..., M_n)) = \log_3 \beta.
$$
This dimension falls in the interval $[0, \log_3 2]$.
\end{thm}

This  construction is quite explicit in the special case $\mathcal{C}(1, M)$.
In that case already the associated graphs $G$ can be very complicated,
and there exist examples where the graph has an arbitrarily large number of
strongly connected components, cf. \cite{ABL13}.

We have computed Hausdorff dimensions of  many
examples of such intersections. In the process we have found some infinite families of integers
where the graph structures are analyzable, see Section \ref{sec4} and \cite{ABL13}.
From the viewpoint of fractal constructions, the sets constructed give specific interesting examples of 
graph-directed fractals, which appear to have structure depending on the integers
$(M_1, .., M_n)$ in an intricate way.

%
%
%

\subsection{Hausdorff dimension results: Two infinite families }\label{sec13a}
There are some simple 
properties of the $3$-adic expansion of $M$ (which coincides with the ternary expansion of $M$, read backwards)
which restrict the Hausdorff dimension of sets $\mathcal{C}(1, M).$
We begin with some simple restrictions on the Hausdorff dimension which
can be read off from the $3$-adic  expansion of $M$; this coincides with the ternary
expansion of $M$, written $(M)_3$, written backwards, where  we write the ternary expansion
$$
(M)_3 := (a_k a_{k-1} \cdots a_1 a_0)_ 3 , \quad\quad {for} \quad M = \sum_{j=0}^k a_j 3^j.
$$
If the first nonzero $3$-adic digit $a_0= 2$, then $\mathcal{C}(1, M)=\{0\}$,
whence   its Hausdorff dimension $\dim_{H}(\mathcal{C}(1, M))=0$. On the other hand, if 
the positive integers $M_1, ..., M_k$ all  all digits $a_j= 0$ or $a_j=1$ in
their $3$-adic expansions, then the Hausdorff dimension 
$\dim_{H}(\mathcal{C}(1, M_1, M_2, ..., M_k))$ must be positive. 

We have found several infinite families  of integers having ternary expansions of
a simple form, whose path set  presentations have a regular structure in the family parameter $k$,
that permits  their Hausdorff dimension to be determined.
The simplest family  takes $M_1= 3^k = (10^k)_3$.  In this trivial case
$\mathcal{C}(1, 3^k) = \Sigma_{3, \bar{2}}$, whence
\begin{equation}
\dim_{H}( \mathcal{C}(1, M_k)) = \log_3 2 \approx 0.630929.
\end{equation}
In Section \ref{sec4} we analyze  two other  infinite families in  detail, as follows.
The first of these families is    $L_k=  \frac{1}{2}(3^{k}-1)  = (1^{k})_3$, for $k \ge 1$.


\begin{thm}\label{th17a} {\rm (Infinite Family $L_k=\frac{1}{2}(3^k-1)$)}

 (1) Let  $L_k=  \frac{1}{2}(3^{k}-1)  = (1^{k})_3$. The path set presentation $(\sG, v_0)$ for 
 the path set $X(1, L_k)$ underlying $\mathcal{C}(1, L_k))$
 has exactly $k$ vertices and is strongly connected. 
 
 (2) For every $k \ge 1$, 
\[
\dim_H(\mathcal{C}(1,L_k) =  \dim_{H} \mathcal{C}(1, (1^k)_3)
= \log_3 \beta_k, 
\]
where $\beta_k $ is the unique real root greater than $1$ of $\lambda^k - \lambda^{k-1}- 1=0$.

(3) For all $k \ge 3$ there holds
\[
\dim_{H} \Big(\mathcal{C}(1,L_k)\Big) = \frac{ \log_3 k}{k} + O\left(\frac{\log\log (k)}{k}\right).
\] 
\end{thm}  

The Hausdorff dimension of the set $\dim_{H}(\mathcal{C}(1,L_k))$ is positive but approaches
$0$ as $k \to \infty$. This result is proved in Section \ref{sec42}.

Secondly, we consider the family  $N_k= 3^{k} + 1 = (10^{k-1}1)_3$.
Our main results concern this family.


\begin{thm}\label{th14} {\rm (Infinite Family $N_k=3^k+1$)}

(1) Let $N_k=3^k+1= (10^{k-1}1)_3$. The path set presentation $(\sG, v_0)$ for 
the path set $X(1, N_k)$ underlying $\mathcal{C}(1, N_k)$
 has exactly $2^k$ vertices and is strongly connected.

(2) For every integer $k \geq 1$, there holds
\[
\dim_H(\mathcal{C}(1,N_k)) =  \dim_{H} \mathcal{C}(1, (10^{k-1}1)_3)
= \log_3\bigg(\frac{1 + \sqrt{5}}{2}\bigg) \approx 0.438018.
\]
\end{thm}  

Here the Hausdorff dimension is constant as $k \to \infty$.
Theorem \ref{th14}  is a direct consequence of  results 
established in Section \ref{sec43} (Theorem \ref{th34} and Proposition \ref{pr45}).

We also include  results on multiple intersections of sets in the two infinite families above
in Section \ref{sec45}.
It is easy to see that for  each infinite family above, the Hausdorff dimensions of
arbitrarily large intersections are always positive. We give some
lower bounds on the dimension; Theorem \ref{th413}
gives multiple intersections  that establish $\Gamma_{\star} \ge \frac{1}{2} \log_3 2.$

In a sequel  \cite{ABL13} 
we analyze a third infinite family $P_k= (20^{k-1}1)_3 = 2 \cdot 3^k +1$,
whose underlying path set graphs  exhibit much more complicated behavior;
they have an 
unbounded number of strongly connected components as $k \to \infty$.

%
%
%

\subsection{Hausdorff dimension results: exceptional sets}\label{sec14a}

In addition we  are able to combine graphs in the infinite family $\mathcal{C}(1, N_k)$ in
such a way to get $\mathcal{C}(1, M_1, M_2, ..., M_n)$ with distinct $M_k \equiv 1 ~(\bmod \, 3)$
which have Hausdorff dimension further bounded away from zero.

In Section \ref{sec51} we establish the following lower bound on the Hausdorff dimension of the generalized
exceptional set.  We are indebted to A. Bolshakov for observing this result, which improves on 
Theorem \ref{th413}.


\begin{thm}\label{th110} 
The generalized exceptional set $\mathcal{E}_{\star}$ satisfies
$$
\dim_{H} ( \mathcal{E}_{\star}) \ge \frac{1}{2} \log_3 2 \approx 0.315464.
$$
In fact,
$$
\dim_{H} ( \{ \lambda \in \Sigma_{3, \bar{2}}: \,   N_{2k+1} \lambda \in \Sigma_{3, \bar{2}} \,\,\mbox{for all}\, k \ge 1\})
\ge \frac{1}{2} \log_3 2.
$$
\end{thm}  

This result is an immediate corollary of Theorem \ref{th51a}.
The proof strongly uses the fact that the integers $N_{2k+1}$ have only two nonzero
$3$-adic digits.

In Section \ref{sec53} we  give numerical improvements  on  the lower bounds in \cite{Lag09} for small $k$ 
for  the Hausdorff dimension of the enclosing sets  $\mathcal{E}^{(k)}(\mathbb{Z}_3)$
that upper bound that of the exceptional set  $\mathcal{E}(\mathbb{Z}_3)$. These improvements come 
 via  explicit examples.

%
%
%

\subsection{Extensions of Results }\label{sec15}

The results of this paper show that the Generalized Exceptional Set $\mathcal{E}_{*}(\ZZ_3)$
has positive Hausdorff dimension. 
Theorem \ref{th110} shows  
that to make further progress on the Exceptional Set Conjecture one cannot relax
the problem to consider general integers $M$; it will be necessary to consider
a smaller class on integers that 
have some special properties in common with  the integers $2^k$.

In a sequel  \cite{ABL13} we investigate another approach towards 
the Exceptional Set Conjecture. 
Let $n_3(M)$ denote the number of nonzero $3$ digits of $M$. 
It asks whether the $\dim_{H} \mathcal{C}(1, M)$ 
necessarily decreases
to $0$ as $n_3(M) \to \infty$. 
It is a  known fact  that the number of nonzero ternary digits in $(2^n)_3$ goes to infinity
as $n \to \infty$, i.e. for each $k \ge 2$ there are only finitely many $n$ with $(2^n)_3$ having
at most $k$ nonzero ternary digits. 
This result was first established in 1971 by Senge and Straus, see  \cite{SS71},
and a quantitative version of this assertion  follows from results of 
C. L. Stewart  \cite[Theorem 1]{St80}. 
It follows that if it  were true that  $\dim_{H} \mathcal{C}(1, M) \to 0$ as $n_3(M) \to \infty$,
then the Exceptional Set Conjecture would follow.

 This paper and its sequel \cite{ABL13} study   the Hausdorff dimension 
of these sets  in the special case of multiplicative translates
of $3$-adic Cantor sets, but one may also consider
 many more complicated path set fractals in the sense of \cite{AL12b}
 in place of the Cantor set. The algorithmic methods of this paper apply to $p$-adic numbers for
 any prime $p$ and to the $g$-adic numbers considered by Mahler \cite{Mah61} for any integer $g \ge 2$.
 
%
%
%

\subsection{Overview}\label{sec16}

Section \ref{sec2} reviews properties of  
$p$-adic path sets and their symbolic dynamics, drawing on \cite{AL12a} and \cite{AL12b}.
The general framework of these papers includes intersections of multiplicative translates of $3$-adic 
Cantor sets as a special case. Section \ref{sec2} also states
a formula for computing the Hausdorff dimension of such sets.
Section \ref{sec3}  of this paper  gives algorithmic constructions and
proves Theorem  \ref{th12}. It also  presents examples. 
Section \ref{sec4}  studies two infinite families of intersections of $3$-adic Cantor sets
 and proves Theorems ~\ref{th17a} and \ref{th14}.
Section \ref{sec5} gives  applications, which include the  lower bound on the Hausdorff dimension of 
the generalized exceptional set $\mathcal{E}_\star(\mathbb{Z}_3)$ and lower bounds on
$\dim_{H}(\mathcal{E}^{(k)}(\mathbb{Z}_3))$ for small $k$.

%
%
%

\subsection{Notation}
The notation $(m)_3$ means either the base $3$ expansion of the positive integer $m$, or else
the $3$-adic expansion of $(m)_3$. In the $3$-adic case this expansion is to be read right to left,
so that it is compatible with the ternary expansion. That is, $\alpha = \sum_{j=0}^{\infty} a_j 3^j$
would be written $( \cdots a_2 a_1 a_0)_3$.\medskip

%
%
%

\section{Symbolic Dynamics and Graph-Directed Constructions} \label{sec2}

%
%
%

\subsection{Symbolic Dynamics, Graphs and Finite Automata } \label{sec21}

The constructions of this paper are based on 
the fact that the points in  intersections of multiplicative translates of $3$-adic Cantor sets
have $3$-adic expansions that are describable in terms of allowable paths generated by
finite directed labeled graphs . 
We  use  symbolic dynamics on certain closed subsets of the one-sided shift space $\Sigma=\sA^{\NN}$ with fixed 
symbol alphabet
$\sA$, which for our application will be specialized to $\sA=\{0,1,2\}$.
A basic reference for directed graphs and symbolic dynamics, 
which we follow, is  Lind and Marcus \cite{LM95}.

By a {\em graph} we mean a finite directed graph, allowing loops and multiple edges. A {\em labeled graph}
is a graph assigning labels  to each directed edge; these labels are drawn from a finite symbol alphabet.
A labeled directed graph can be interpreted as a {\em  finite automaton}
in the sense of automata theory.
In our applications to $3$-adic digit sets, the labels are drawn from the alphabet $\sA= \{ 0, 1, 2\}.$ In a directed
graph, a vertex is a {\em source} if all directed edges touching that vertex are outgoing; it is a {\em sink} if all
directed edges touching that edge are incoming. A vertex  is {\em essential} if it is neither a source nor a sink,
and is called {\em stranded} otherwise. A graph is \emph{essential} if all of its vertices are essential. 
 A  graph $G$ is
  {\em strongly connected} if for each two vertices $i, j$ there is a directed path from $i$ to $j$.
We let $SC(G)$ denote the set of strongly connected component subgraphs of $G$.

We use some basic facts from Perron-Frobenius theory of nonnegative matrices.
The {\em Perron eigenvalue} (\cite[Definition 4.4.2]{LM95}) of a  nonnegative  real matrix $\mathbf{A} \neq 0$ is
the largest real  eigenvalue $\beta \ge 0$ of $\mathbf{A}$.  
A nonnegative matrix is {\em irreducible} if for each row and column $(i, j)$ some power ${\bf A}^m$
has $(i,j)$-th entry nonzero.  A nonnegative matrix ${\bf A}$ is {\em primitive}  if some power ${\bf A}^k$
for an integer $k \ge 1$ has all entries positive; primitivity implies irreducibility but not vice versa.
The {\em Perron-Frobenius theorem,} \cite[Theorem 4.2.3]{LM95} for
an irreducible nonnegative matrix ${\bf A}$ states that:
\begin{enumerate}
\item
The Perron eigenvalue $\beta$ is geometrically
and algebraically simple, and has
 an everywhere positive eigenvector ${\bf v}.$
 \item
 All other eigenvalues $\mu$ have $|\mu| \le \beta$, so that $\beta = \sigma({\bf A})$,
 the spectral radius of ${\bf A}$.
 \item 
 Any other everywhere positive eigenvector must 
 be  a positive mulitiple of ${\bf v}$.
 \end{enumerate}
 For a general nonnegative 
real matrix $\mathbf{A} \neq 0$, the Perron eigenvalue need not be simple, but it
still equals the spectral radius $\sigma(\bf{A})$ and it has at least one  everywhere nonnegative eigenvector.

We apply this theory to adjacency matrices of graphs. 
A (vertex-vertex) {\em adjacency matrix} ${\bf A} ={\bf A}_{G}$ of the directed graph  $G$ has
entry $a_{ij}$ counting the number of directed edges from vertex $i$ to vertex $j$.
The adjacency matrix is irreducible if and only if the associated graph is strongly connected,
and we also call the graph {\em irreducible} in this case.
  Here primitivity of the adjacency matrix of 
a directed  graph $G$ is equivalent to the graph being strongly connected
and aperiodic, i.e.  the greatest common divisor of its (directed) cycle lengths is  $1$. 
For an adjacency matrix of a graph containing at least
at least one directed cycle,  its Perron eigenvalue is necessarily a real algebraic integer $\beta \ge 1$
(see Lind \cite{Lin84} for a characterization of these numbers).

%
%
%

\subsection{$p$-Adic path sets, sofic shifts  and $p$-adic path set fractals} \label{sec21b}

Our basic objects are special cases of the following definition. 
A {\em pointed graph} is a pair $(\sG, v)$ consisting of a directed labeled graph $\sG =(G, \mathcal{E})$
and a  marked vertex $v$ of $\sG$. Here $G$ is a (directed) graph and $\mathcal{E}$ is 
an assignment of labels $(e, \ell)= (v_1, v_2, \ell)$ to the edges of $G$, where every edge gets a unique label,
and no two triples are the same (but multiple edges and loops are permitted otherwise).

\begin{defn} \label{de211}
 Given a pointed graph $(\sG, v)$ its associated
 \emph{path set} 
   $\sP = X_\mathcal{G}(v) \subset \sA^{\NN}$ is
the set of  all infinite one-sided symbol sequences $(x_0, x_1, x_2, ...) \in \sA^{\NN}$, 
giving the successive labels of all one-sided infinite walks in $\mathcal{G}$ 
issuing from the distinguished vertex $v$. 
Many different $(\mathcal{G}, v)$ may give the same path set $\sP$, and we call any such
$(\mathcal{G},v)$  a \emph{presentation} of $\sP$. 
\end{defn}

An important class of presentations have the following extra property.
We  say that a directed labeled graph $\sG =(G, v)$ is {\em right-resolving}
if for each vertex of $\sG$ all directed edges outward have
distinct labels. (In automata theory $\sG$ is called a {\em deterministic
automaton}.) One can show that every path set has a right-resolving presentation.

Note that the labeled  graph $\sG$ without a marked vertex determines
a {\em one-sided sofic shift} in the sense of symbolic dynamics, as defined in \cite{AL12a}. 
This sofic shift  comprises
the set union of the path sets at all vertices of $\sG$.
Path sets are closed sets in the shift topology,
but are in general non-invariant under the one-sided shift operator.
Those path sets $\sP$  that are invariant are exactly the one-sided sofic shifts \cite[Theorem 1.4]{AL12a}. 

We study the path set concept in symbolic dynamics in \cite{AL12a}.
 The collection of path sets $X:= X_{(\sG, v_0)}$ in a given
alphabet is closed under finite union and intersection (\cite{AL12a}). 
The symbolic dynamics analogue of Hausdorff dimension is topological entropy. 
The {\em topological entropy} of a path set $H_{top} (X)$ is given by
$$
H_{top}(X) := \limsup_{n \to \infty} \frac{1}{n} \log N_n(X),
$$
where $N_n(X)$ counts the number of distinct blocks of symbols of lengh $n$
appearing in elements of $X$. 
The topological entropy is easy to compute for right-resolving presentation.
By \cite[Theorem 1.13]{AL12a}, it is 
 \begin{equation} \label{top-entropy}
 H_{top}(X) = \log \beta
 \end{equation}
 where 
 $\beta$ is the Perron eigenvalue of the adjacency matrix ${\bf A}={\bf A}_G$ of 
 the underlying directed graph $G$ of $\sG$, e.g. the spectral radius of ${\bf A}$.

%
%
%

\subsection{$p$-Adic Symbolic Dynamics and Graph Directed Constructions}\label{sec23a}

We now suppose $\sA = \{0, 1,2, ..., p-1\}$.
We can view the elements  of 
a path set $X$ on this alphabet  geometrically as describing the digits in the
$p$-adic expansion of a $p$-adic integer.  This is done using a map
$\phi: \sA^{\NN} \to \ZZ_p$.
from symbol sequences into $\ZZ_p$.
We call the resulting image set $K = \phi(X)$ a \emph{$p$-adic path set fractal}.
Such sets are studied in \cite{AL12b}, where they are related to
graph-directed fractal constructions.
The class of  $p$-adic path set fractals  is closed under  $p$-adic addition
and multiplication by rational numbers $r \in \QQ$ that lie in $\ZZ_p$ (\cite{AL12b}).

It is possible to compute the Hausdorff dimension of a 
$p$-adic path set fractal  directly from a suitable presentation 
of the underlying path set $X=X_{\sG}(v)$.
We will use the following result.
\begin{prop}\label{pr22a}
Let $p$ be a prime, and $K$ a set of $p$-adic integers whose allowable
$p$-adic expansions are described by the symbolic dynamics
of a $p$-adic path set  $X_K$ on symbols $\mathcal{A} =\{ 0, 1, 2, \cdots, p-1\}$.
Let $(\mathcal{G},v_0)$ be a presentation of
this path set that is right-resolving.

(1)  The map $\phi_p: \mathbb{Z}_p \rightarrow [0,1]$ taking 
$\alpha= \sum_{k=0}^{\infty}{a_k p^k} \in \mathbb{Z}_p$
to the real number with base $p$ expansion
 $\phi_p(\alpha) :=\sum_{k=0}^\infty \frac{a_k}{p^{k+1}}$
is a continuous map, and the image of $K$ under this map,  $K':= \phi_p(K) \subset [0,1]$, is a
graph-directed fractal in the sense of Mauldin-Williams.

 (2) The Hausdorff dimension of the $p$-adic path set fractal $K$ is 
\begin{equation}
\dim_{H}(K) = \dim_{H}(K') = \log_p  \beta,
\end{equation}
where $\beta$ is the spectral radius of the adjacency matrix ${\bf A}$ of $G$.
\end{prop}

\begin{proof}
These results are  proved in \cite[Section 2]{AL12b}.
\end{proof}

In this paper we treat the case $p=3$ with $\sA = \{ 0, 1, 2\}$.
The $3$-adic Cantor set is a $3$-adic path set fractal, so these
general properties above guarantee that the intersection of 
a finite number of multiplicative translates
of $3$-adic Cantor sets will itself be a $3$-adic path set fractal $K$,
generated from an underlying path set. 

To do calculations with such sets we will need algorithms
for converting presentations of a given $p$-adic path set to presentations
of new $p$-adic path sets derived by the operations above. 
The $p$-adic arithmetic operations are treated in \cite{AL12b}
and union and intersection are treated in \cite{AL12a}.

%
%
%

\section{Structure of Intersection Sets $\mathcal{C}(1,M_1, M_2, ..., M_n$)} \label{sec3}

We  show that the sets $C(1, M_1,\ldots , M_n)$ consist of those $3$-adic integers whose $3$-adic
expansions are describable as path sets $X(1, M_1, \cdots , M_n)$. We also present an algorithm which when given the data 
$(M_1, ...,M_n)$ as input produces as output a presentation $\sG = (G, v_0)$ of the path set $X(1,M_1,\ldots,M_n)$.

%
%
%

\subsection{Constructing a path set presentation  $X(1, M)$}\label{sec31} 
We describe an algorithmic procedure to obtain a  path set presentation $X(1, M)$ for the $3$-adic expansions of
elements in $\mathcal{C}(1, M)$. Since $\mathcal{C}(1, 3^j M) = \mathcal{C}(1, M)$, 
we may reduce to the case $M \not\equiv 0~(\bmod ~3)$ and since $\mathcal{C}(1, M) = \{0\}$
if $M \equiv 2 ~(\bmod\, 3)$ it suffices to consider the case $M \equiv 1 \, (\bmod \, 3)$.


\begin{thm}\label{th31n}
For $M \ge 1$, with $M \equiv 1 ~(\bmod \, 3)$, the set $\mathcal{C}(1, M) = \Sigma_3 \cap \frac{1}{M} \Sigma_3$ 
has $3$-adic expansions given by a path set $X(1, M)$ which 
has an algorithmically computable path set presentation  $(\mathcal{G}, v_0)$, in which the vertices $v_{m}$ 
are labeled with a subset of the integers $0 \le m \le \lfloor \frac{1}{2} M \rfloor$, always including $m=0$,
and of cardinality at most $\lfloor \frac{M}{2}\rfloor$.
This presentation is right-resolving, connected and
essential. 
\end{thm}


\begin{proof}
The labeled graph $\mathcal{G} = (G, \mathcal{L})$ will have path labels drawn from $\{0, 1\}$
and the vertices $v_j$ of the underlying directed graph $G$ will be labeled by 
a subset of the integers $j$ satisfying
$0 \le N \le  M+1.$ 
The marked vertex $v_0$ corresponds to $N=0$
and is the starting vertex of the algorithm.

The idea is simple. Suppose that
$$
\alpha:= \sum_{j=0}^{\infty} a_j 3^j \in \Sigma_3 \cap \frac{1}{M} \Sigma_3.
$$
Here all $a_j \in \{0, 1\}$ and in addition 
$$
M \alpha = \sum_{j=0}^{\infty} b_j 3^j \in \Sigma_3.
$$
Suppose the first $n$ digits 
$$\alpha_n = \sum_{j=0}^{n-1} a_j 3^j,$$
are chosen.
Since $M \equiv 1 ~(\bmod \, 3)$ this uniquely specifies the first $n$ digits of
$$
M \alpha_n := \sum_{j=0}^{m+n-1} b_j^{(n)} 3^j,
$$
namely
$$
b_j^{(n)} = b_j \,\, \mbox{for} \quad 0 \le j \le n-1,
$$
which have $b_j \in \{0, 1\},$ for $0 \le j \le n-1.$
Here the  remaining digits $b_{n+k}^{(n)}$ for $1 \le k \le m$ are unrestricted, with
$$
 m =\lfloor  \log_3 M \rfloor +1.
$$
 We have followed a path in the graph G corresponding
to edges labeled $(a_0, a_1, ..., a_{n-1})$. The vertex we arrive at after these steps will be
labeled by the value of the ``carry-digit"  part of $\beta_n$, which is
$$
N= \sum_{j =n}^{m+n -1} b_j^{(n)} 3^{j-n}.
$$

The value of the bottom $3$-adic digit $b_n^{(n)}$ of $N$ will determine the allowable exit 
edges from vertex $v_N$, and the label of the vertices reached. The requirement
is that the next digit $a_{n}$ satisfy
\begin{equation}\label{allowable0}
a_n + b_n^{(n)}  \equiv  0, 1 (\bmod \, 3)
\end{equation}
If such a value is chosen, then we will be able to create a valid $\alpha_{n+1}$
and $\beta_{n+1} := M \alpha_{n+1}$ will have 
$$
b_{n}^{(n+1)} = a_n + b_n^{(n)} ~(\bmod \, 3).
$$
There always exists at least one exit edge from each reachable vertex $v_N$, 
since for $b_n^{(n)}=0$ the admissible $a_n =0, 1$; for $b_n^{(n)}=1$ the only admissible $a_n = 0$,
and for $b_n^{(n)} =2$ the only admissible $a_n=1$, in order that the next digits
$a_{n+1}, b_{n+1}$ both belong to $\{0, 1\}$.

The important  point  is that the  vertex label $N$ is all that must be remembered to decide on an
admissible  exit edge in the next step, since its bottom digit determines
the allowable exit edge values $a \subset \{0, 1\}$ by requiring
\begin{equation}\label{allowable}
a+ N \equiv 0, 1 \, ~(\bmod \, 3),
\end{equation}
and  for an exit edge labeled $a$ one can determine the  new vertex  label $v_{N'}$ as
\begin{equation} \label{update-step}
N' := \lfloor \frac{N + M a}{3} \rfloor. 
\end{equation}
To the graph $G$ one adds a directed edge  for each allowable value $a_n=0$ or $1$ from $N$ to $N'$ labeled by $a_n$.

Now one sees that the are only finitely many vertices $v_N$ that can be reached from the vertex $v_0$.
One proves by induction on the number of steps $n$ taken that any reachable vertex 
 $v_N$ has vertex label.
 $$ 
 0 \le N \le \lfloor \frac{M}{2} \rfloor.
$$
This holds for the initial vertex, while
for the induction step, we obtain from \eqref{update-step} that
$$
N' \le \frac{N+ Ma}{3} \le \frac{ M/2 + M}{3} \le \frac{M}{2}.
$$ 
Thus the process of constructing the graph will halt.

It is easily seen that the presentation $\mathcal{G}= (G, v_0)$
obtained this way  has the desired properties.
\begin{enumerate}
\item[(1)]
The graph $G$ is right-resolving because there every vertex has exit edges with
distinct edge-labels by construction. 
\item[(2)]
The graph $G$ is essential because every vertex has at least one admissible exit edge,
as shown above.
\item[(3)]
The graph is connected since we  include in it only vertices reachable from $v_0$.
\end{enumerate}
Since $G$ is essential, $\mathcal{G}$ is a presentation of a certain $3$-adic path set via 
the correspondence taking infinite walks beginning at the $v_0$-state in $\mathcal{G}$ to words in the edges traversed. 
Denote this path set $X_{\mathcal{G},0}$. 

It remains to prove  that this is the path set $X(1,M)$ corresponding to $\mathcal{C}(1, M)$, which is the claim that 
$$
X_{\mathcal{G},0} = X(1,M).
$$
To prove the claim, let  $\Phi : X_{\mathcal{G},0} \rightarrow \mathbb{Z}_3$ be the map 
\[\cdots a_2 a_1 a_0 \mapsto \sum_{k=0}^{\infty}{a_k3^k}.\] $\Phi$ is clearely an injection.
$\Phi(X_{\mathcal{G},0}) \subset \mathcal{C}(1,M)$:  Since $ \cdots a_2 a_1 a_0 \in X_{\mathcal{G},0}$ 
is a word in the full shift on $\{0,1\}$, $\Phi(\cdots a_2 a_1 a_0) = \sum_{k=0}^{\infty}{a_k3^k}$
 omits the digit 2, so that $\Phi(X_{\mathcal{G},0}) \subset \Sigma_3$. But the algorithm was 
 constructed specifically so that, given a path $\pi = a_l a_{l-1} \cdots a_2 a_1 a_0$ in $\mathcal{G}$ 
 originating at 0, there is an edge labeled $a_{l+1} \in \{0,1\}$ from the terminal vertex $t(\pi)$ 
 if and only if each digit of the 3-adic expansion of $M \cdot \big(\sum_{k=0}^{l+1}{c_k3^k}\big)$
  which cannot be altered by any potential $(l+2)$nd digit is either 0 or 1. 
This shows both that $\Phi(X_{\mathcal{G},0}) \subset \frac{1}{M} \Sigma_3$ and 
$\mathcal{C}(1,M) \subset \Phi(X_{\mathcal{G},0})$, so that 
$\Phi|_{\Phi^{-1}(\mathcal{C}(1,M))} : X_{\mathcal{G},0} \rightarrow \mathcal{C}(1,M)$ 
is a bijection. Assigning the appropriate metric to $X_{\mathcal{G},0}$ makes $\Phi$ 
an isomorphism in a now obvious way, proving the claim. 
\end{proof}

We obtain an algorithm to construct $\mathcal{G}= (G, v_0)$ based on the construction above.\medskip

\noindent {\bf Algorithm A} {\em (Algorithmic Construction of Path Set Presentation  $X(1, M)$).}

\begin{enumerate}
\item
(Initial Step) Start with initial marked vertex $v_0$, and initial vertex
set  $I_{0} := \{ v_0\}$.  Add an  exit edge with edge label $0$
giving a self-loop to $v_0$, and add 
 another  exit edge with edge label $1$ going to new vertex $v_m$ with vertex label 
 $m := \lfloor M/3 \rfloor,$
Add these two edges and their labels to form (labeled) edge table $E_1$. Form the new vertex set 
$I_{1} := \{ v_m\}$,  and go to Recursive Step with $j=1$. 
\item
(Recursive step) Given value $j$, a nonempty new vertex set $I_j$ of level $j$ vertices, a
current vertex set $V_j$ and current edge set $E_j$.  At step $j+1$ determine  all 
allowable exit edge labels  from vertices $v_N$ in $I_j$, using the criterion \eqref{allowable}, 
 and compute vertices reachable by these exit edges, with
reachable vertex labels computed by update equation \eqref{update-step}.
Add these new edges and their labels to current edge set  to make updated current edge set $E_{j+1}$.
Collect all  vertices reached that are not  in current vertex set $V_j$  into a new vertex set
 $I_{j+1}$.  Update current vertex set $V_{j+1} = V_j \cup I_{j+1}.$
Go to test step.
\item
 (Test step).
 If the current vertex set $I_{j+1}$ is empty, halt, with the complete presentation $\sG = (G, v_0)$
 given by sets $V_{j+1}, E_{j+1}$.
  If  $I_{j+1}$  is nonempty, reset $j \mapsto j+1$ and go to  Recursive Step.
 \end{enumerate}
 
 The correctness of the algorithm follows from the discussion above.
 
%
%
%

\subsection{Constructing a path set presentation $X(1,M_1,\ldots,M_n)$} \label{sec31b}

Given integers $1 \leq M_1 < \ldots < M_n$, we now have a way to construct 
graph presentations of the path sets  $X(1,M_i)$ for each $i$. Since 
\[X(1,M_1,\ldots,M_n) = \bigcap_{i=1}^n X(1,M_i),\]
we need to know how to combine these graphs.

Recall the following definition from Lind and Marcus \cite{LM95}:
\begin{defn}  Let $\mathcal{G}_1$ and $\mathcal{G}_2 $ be labeled graphs with the same 
alphabet $\mathcal{A}$, and let their underlying graphs be $G_1 = (\mathcal{V}_1,\mathcal{E}_1)$ 
and $G_2 = (\mathcal{V}_2,\mathcal{E}_2)$. The label product 
$\mathcal{G}_1 \star \mathcal{G}_2$ of $\mathcal{G}_1$ and $\mathcal{G}_2$ 
has underlying graph $G$ with vertex set $\mathcal{V} = \mathcal{V}_1 \times \mathcal{V}_2$, 
edge set $\mathcal{E} = \{(e_1,e_2) \in \mathcal{E}_1 \times \mathcal{E}_2 : e_1 \text{ and } e_2 \text{ have the same labels}\}$.
\end{defn}

In \cite[Proposition 4.3]{AL12a}, we show that if $(\mathcal{G}_i,v_i)$ is a graph presentation
 of the path set $\mathcal{P}_i$, then $(\mathcal{G}_1 \star \mathcal{G}_2,(v_1,v_2))$ is a 
 graph presentation for $\mathcal{P}_1 \cap \mathcal{P}_2$. It follows that we can form a 
 presentation of $\mathcal{C}(1,M_1, \cdots , M_n)$ as the label product 
\[(\mathcal{G},v) = (\mathcal{G}_1 \star \mathcal{G}_2 \star \cdots \star \mathcal{G}_n, (v_1,v_2,\ldots, v_n)),\] 
where $(\mathcal{G}_i,v_i)$ is the presentation of $\mathcal{C}(1,M_i)$ just constructed.


\begin{thm}\label{th33n}
For $1 < M_1<  M_2 < \cdots < M_n$, with all $M_i \equiv 1 ~(\bmod \, 3)$, the set 
$$
\mathcal{C}(1, M_1, M_2, \cdots, M_n) ) = \bigcap_{i=1}^n \mathcal{C}(1, M_i) = 
\Sigma_3 \cap (\bigcap_{i=1}^n \frac{1}{M_i} \Sigma_3),
$$ 
has $3$-adic expansions of its elements given by a path set
$X(1, M_1, M_2, \cdots, M_n)$.
This path set has an algorithmically computable
 presentation  $(\mathcal{G}, v_{\bf 0})$, in which the vertices $v_{\bf{N}}$ 
are labeled with a subset of integer vectors ${\bf{N}} =(N_1, N_2, ..., N_n)$ with $0 \le N_i \le \frac{1}{2} M_i$, always including 
the zero vector $\bf{0}$.  The presentation has at most $\prod_{i=1}^n ( 1+ \lfloor \frac{1}{2}M_i \rfloor)$ vertices in the underlying graph.
This presentation is right-resolving, connected and essential. 
\end{thm}

\begin{proof} The presentation is obtained by recursively applying the label product construction to
the presentations $\mathcal{C}(1, M_i)$, see Algorithm B below.
Each step  preserves the properties
of the presentation graph being right-resolving, connected and  essential. 
The number of states of the label product construction is
at most the product of the number of states in the two presentations being constructed.
By  Theorem \ref{th31n}, the presentation of $\mathcal{C}(1, M_i)$ has at most
  $(1 + \lfloor \frac{1}{2} M\rfloor )$ vertices. 
The bound given follows by induction on the 
successive label product constructions. 
\end{proof}

\noindent{\bf Algorithm B} {\em (Algorithmic Construction of Path Set Presentation  $X(1, M_1, ..., M_n)$.}

\begin{enumerate}
\item
(Initial Step) 
Construct presentations $\mathcal{G}_i = (G_i, \sL_i)$ for $X(1, M_i)$ to $\mathcal{C}(1, M_i)$
for $1 \le i \le n$, using Algorithm A.  Apply the label product construction to form
$\mathcal{H}_2 := \mathcal{G}_1 \star \mathcal{G}_2$.

\item
For  $2 \le i \le n-1$,
apply the label product construction to form
$$\mathcal{H}_{i+1}=\mathcal{H}_{i} \star \mathcal{G}_{i+1}.$$
Halt when $\mathcal{H}_n$ is computed.
 \end{enumerate}

%
%
%

\subsection{Path Set Characterization of $\mathcal{C}(1, M_1, ..., M_n)$} \label{sec32}

From Theorem \ref{th33n} we easily derive the following result.
\begin{thm} \label{th11}
For any integers $1\leq M_1<\ldots<M_n$, let  
$$\mathcal{C}(1,M_1,\ldots,M_n):=
\Sigma_{3} \cap \frac{1}{M_1}\Sigma_{3} \cap \ldots \cap \frac{1}{M_n}\Sigma_{3}.
$$
This is the set of all $3$-adic integers  $\lambda \in \Sigma_{3}$ such that $M_j \lambda$ omits the digit $2$
in its $3$-adic expansion. Then:

(1) The complete set of the $3$-adic expansions of numbers in the set $\mathcal{C}(1,M_1,\ldots,M_n)$,
 is a path set in the alphabet $\sA= \{ 0, 1, 2\}.$

 (2) The Hausdorff dimension of $\mathcal{C}(1,M_1,\ldots,M_n)$
 is $\log_3 \beta$, where $\log \beta$ is the topological entropy of this path set. Here $\beta$ necessarily satisfies
 $1 \le \beta \le 2$, and  $\beta$ is
 a Perron number, i.e. it is a real algebraic integer $\beta \ge 1$ such that all its
 other  algebraic conjugates satisfy $|\sigma(\beta)| < \beta.$ 
 \end{thm}

\begin{proof}
Theorem~\ref{th33n} gives an explicit construction 
of a presentation $(\sG, v)$ showing that $\mathcal{C}(1,M_1\ldots,M_n)$ is a $p$-adic path set.

By Proposition ~\ref{pr22a}
 the Hausdorff dimension of $\mathcal{C}(1,M_1\ldots,M_n)$ is $\log_3 \beta$, where $\beta$ is the spectral
 radius  of the adjacency matrix $A$ of the underlying graph $G$. 
  Since $A$ is a 0-1 matrix, by Perron-Frobenius theory the spectral radius equals the maximal eigenvalue
 in absolute value, which is necessarily a positive real number
 $\beta$. It is a solution to a monic polynomial over $\mathbb{Z}$, so that $\beta$ is necessarily an algebraic integer. 
 By construction, the sum of the entries of any row in $A$ is either 1 or 2, so that we also have $1 \leq \beta \leq 2$. 
\end{proof}

\begin{rem} The adjacency matrix $A$ in the sets above {\em need not be irreducible}. 
Example \ref{example33} 
below presents a graph $\mathcal{C}(1, 19)$ having a reducible matrix $A$.
Here the  underlying graph $\sG$ has  two strongly connected components. 
 \end{rem}

Combining the results above establishes  Theorem \ref{th12}.


\begin{proof}[Proof of Theorem \ref{th12}]
(1) 
This follows from Theorem \ref{th31n} and Theorem \ref{th33n}, 
with the algorithm for constructing a 
the presentation of the path set $X(1, M_1, M_2, \cdots , M_n)$ given by combining 
 {Algorithm A} and { Algorithm B}.

(2) This follows from Theorem \ref{th11}. 
\end{proof}

%
%
%

\subsection{Examples}\label{sec33}

We present several examples of path set presentations. 

%
\begin{exmp}\label{example31}
The  $3$-adic Cantor set $\Sigma_3 = \mathcal{C}(1) = \mathcal{C}(1, 1)$ 
has a path set presentation  $ (\sG, v_0)$ pictured in Figure \ref{fig31}. It is the full shift on
two symbols, and the initial vertex is the vertex labeled $0$. The underlying graph $G$  of $\sG$ is a double cover
of a one vertex graph with two symbols. The advantage of the graph $G$ pictured is that
a path for it is completely determined by the set of vertex symbols that it passes through.

%

\begin{figure}[ht]\label{fig31}
	\centering
	\psset{unit=1pt}
	\begin{pspicture}(-80,-40)(80,40)
		\newcommand{\noden}[2]{\node{#1}{#2}{n}}
		\noden{0}{-35,0}
		\noden{1}{35,0}
		\bcircle{n0}{90}{0}
		\dline{n0}{n1}{1}{0}
		\bcircle{n1}{270}{1}
	\end{pspicture}
		\newline
\hskip 0.5in {\rm FIGURE 3.1.} Path set presentation of Cantor shift $\Sigma_3=\mathcal{C}(1)$. The marked vertex is  $0$.
\newline
\newline
\end{figure}

\end{exmp}
\pagebreak

\begin{exmp}\label{example32}
A path set presentation of $\mathcal{C}(1,7)$, with $7=(21)_3$ is shown  in Figure \ref{fig32}.
The vertex labeled $0$ is the marked initial state. 

%
%
%

\begin{figure}[ht]\label{fig32}
	\centering
	\psset{unit=1pt}
	\begin{pspicture}(-80,-50)(80,150)
		\newcommand{\noden}[2]{\node{#1}{#2}{n}}
		\noden{0}{0,100}
		\noden{1}{50,50}
		\noden{2}{-50,50}
		\noden{10}{0,0}
		\bcircle{n0}{0}{0}
		\bcircle{n10}{180}{1}
		\bline{n0}{n2}{1}
		\bline{n2}{n10}{1}
		\bline{n10}{n1}{0}
		\bline{n1}{n0}{0}
	\end{pspicture}
	\newline
\hskip 0.5in {\rm FIGURE 3.2.} Path set presentation of $\mathcal{C}(1,7)$. The marked vertex is $0$.
\end{figure}

The graph in Figure  \ref{fig32} has adjacency matrix 
\begin{equation*}
\bf{A} = \left(\begin{array}{cccc}
1 & 1 & 0 & 0 \\
0 & 0 & 1 & 0 \\
0 & 0 & 1 & 1 \\
1 & 0 & 0 & 0 \\
\end{array}\right),
\end{equation*}
which has Perron-Frobenius eigenvalue $\beta = \frac{1 + \sqrt{5}}{2}$, so 
\[\dim_H(\mathcal{C}(1,7)) = \log_3 \left( \frac{1 + \sqrt{5}}{2} \right) \approx 0.438018.
\]
\end{exmp}


\begin{exmp}\label{example33}
A path  set presentation of $\mathcal{C}(1,19)$, with $19 = (201)_3$,  is shown in 
Figure \ref{fig33}. The node labeled $0$ is the marked initial state.

%
%
%

\begin{figure}[ht]\label{fig33}
 \centering
 \psset{unit=1pt}
 \begin{pspicture}(-100,-165)(100,160)
  \newcommand{\noden}[2]{\node{#1}{#2}{n}}
  \noden{0}{0,110}
  \noden{1}{50,55}
  \noden{10}{50,-55}
  \noden{100}{0,-110}
  \noden{22}{-50,-55}
  \noden{20}{-50,55}
  \noden{21}{0,-25}
  \noden{2}{0,25}
  \bcircle{n0}{0}{0}
  \bcircle{n100}{180}{1}
  \dline{n2}{n21}{1}{0}
  \aline{n0}{n20}{1}
  \aline{n20}{n22}{1}
  \aline{n22}{n100}{1}
  \aline{n100}{n10}{0}
  \aline{n20}{n2}{0}
  \bline{n1}{n0}{0}
  \bline{n10}{n21}{1}
  \bline{n10}{n1}{0}
 \end{pspicture}
 \newline
 \hskip 0.5in {\rm FIGURE 3.3.} Path set presentation of $\mathcal{C}(1,19)$. The marked vertex is $0$.

\end{figure}

The graph in Figure \ref{fig33} has
 adjacency matrix \\
\begin{equation*}
\bf{A} = \left(\begin{array}{cccccccc}1 & 1 & 0 & 0 & 0 & 0 & 0 & 0 \\0 & 0 & 1 & 1 & 0 & 0 & 0 & 0 \\0 & 0 & 0 & 0 & 1 & 0 & 0 & 0 \\0 & 0 & 0 & 0 & 0 & 1 & 0 & 0 \\0 & 0 & 0 & 0 & 1 & 0 & 1 & 0 \\0 & 0 & 0 & 1 & 0 & 0 & 0 & 0 \\0 & 0 & 0 & 0 & 0 & 1 & 0 & 1 \\1 & 0 & 0 & 0 & 0 & 0 & 0 & 0 \\\end{array}\right),\end{equation*}
which has 
Perron eigenvalue $\beta \approx 1.465571$, so
\[\dim_H(\mathcal{C}(1,19)) = \log_3 \beta \approx 0.347934.\]
\end{exmp}


\begin{exmp}\label{example34}
We consider implementation of  the algorithm for $\mathcal{C}(1,7,19)$. 
We start from  the presentations of $\mathcal{C}(1, 7)$ and $\mathcal{C}(1, 19)$ in Example \ref{example31}.
Taking the label product gives us a presentation of $\mathcal{C}(1,7,19)$, which is 
shown in Figure \ref{fig34}.

%
%
%
\begin{figure}[ht]\label{fig34}
	\centering
	\psset{unit=1.3pt}
	\begin{pspicture}(-50,-45)(50,155)
	\newcommand{\nodeq}[2]{\node{#1}{#2}{q}}
	\nodeq{0-0}{0,115}
	\nodeq{2-20}{-35,75}
	\nodeq{10-22}{-35,40}
	\nodeq{10-100}{0,0}
	\nodeq{1-10}{35,40}
	\nodeq{0-1}{35,75}
	\bcircle{q0-0}{0}{0}
	\bline{q0-0}{q2-20}{1}
	\bline{q2-20}{q10-22}{1}
	\bline{q10-22}{q10-100}{1}
	\bline{q10-100}{q1-10}{0}
	\aline{q1-10}{q0-1}{0}
	\bline{q0-1}{q0-0}{0}
	\bcircle{q10-100}{180}{1}
	\end{pspicture}
	\newline
\hskip 0.5in {\rm FIGURE 3.4.} Path set presentation of $\mathcal{C}(1,7,19)$. The marked vertex is $0$.
\end{figure}

This graph $G$ for $\mathcal{C}(1, 7, 19)$ has adjacency matrix $\bf{A}$  given by:

\begin{equation*}
\bf{A} = \left(\begin{array}{cccccc}
1 & 1 & 0 & 0 & 0 & 0 \\
0 & 0 & 1 & 0 & 0 & 0 \\
0 & 0 & 0 & 1 & 0 & 0 \\
0 & 0 & 0 & 1 & 1 & 0 \\
0 & 0 & 0 & 0 & 0 & 1 \\
1 & 0 & 0 & 0 & 0 & 0 \\
\end{array}\right).
\end{equation*}

The Perron eigenvalue $\beta \approx 1.46557$ of this matrix  is the largest real root of
 $\lambda^6 -2 \lambda^5 +\lambda^4 -1 = 0$: 
The Hausdorff dimension of $\mathcal{C}(1,7,19)$ is then
\begin{equation} 
\dim_H(\mathcal{C}(1,7,19)) = \log_3 \beta \approx 0.347934.
\end{equation}
\end{exmp}

\begin{exmp} 
The set $\mathcal{C}(1,43)$, with $N= 43= (1121)_3$
has $M \equiv \, 1\, (\bmod \, 3)$, but nevertheless has Hausdorff dimension $0$. A  presentation
of the path set associated to  $\mathcal{C}(1,43)$ is given in Figure \ref{fig35}.


\begin{figure}[ht]\label{fig35}
	\centering
	\psset{unit=1.3pt}
	\begin{pspicture}(-80,-15)(80,105)
	\newcommand{\nodeq}[2]{\node{#1}{#2}{q}}
	\nodeq{0}{0,90}
	\nodeq{112}{-45,90}
	\nodeq{201}{-45,55}
	\nodeq{20}{-45,0}
	\nodeq{121}{0,0}
	\nodeq{2}{0,25}
	\nodeq{12}{45,0}
	\nodeq{120}{45,55}
	\bcircle{q0}{270}{0}
	\bline{q0}{q112}{1}
	\bline{q112}{q201}{1}
	\bline{q201}{q20}{0}
	\bline{q20}{q121}{1}
	\aline{q20}{q2}{0}
	\dline{q121}{q12}{0}{1}
	\aline{q2}{q120}{1}
	\aline{q120}{q12}{0}
	\bline{q120}{q201}{1}
	\end{pspicture}
	\bigskip
	\newline
\hskip 0.5in {\rm FIGURE 3.5.} Path set presentation of $\mathcal{C}(1,43)$. The marked vertex is $0$.
\newline
\newline
\end{figure}

The  graph in Figure \ref{fig35} has four strongly connected components, with vertex sets
$\{0\}, \{ 112\}, \{ 2, 120, 201, 20\},$ and $\{ 12, 121\}$  respectively, each of whose underlying path sets have
Hausdorff dimension $0$.
\end{exmp}

\newpage

%
%
%

\section{Infinite Families}\label{sec4}
%
%
%

\subsection{Basic Properties}\label{sec41}

We have  the following simple result, showing the influence of the digits in
the $3$-adic expansion of $M$ on the size of the set $\mathcal{C}(1,M)$
and $\mathcal{C}(1, M_1, M_2, \cdots, M_k)$.

\begin{thm} \label{th31}
(1) If the smallest nonzero $3$-adic digit in the $3$-adic expansion
of the positive integer $M$ is $2$, then $\mathcal{C}(1, M)= \{0\}$,
and
\begin{equation}
\dim_{H}( \mathcal{C}(1,M))=0.
\end{equation}

(2) If positive integers $M_1, M_2 , ..., M_n \in \Sigma_3$ all
have the property  that   their $3$-adic expansions 
$(M_i)_3$ (equivalently their ternary expansions) contain only digits $0$ and $1$, then
\begin{equation}
\dim_{H} (\mathcal{C}(1, M_1, M_2, ..., M_n)) >0.
\end{equation}
 \end{thm}
 
 \paragraph{\bf Remark.}

For neither (1) or (2) does the converse hold. 
The example   $M=43= (1121)_3$ has $\dim_{H}( \mathcal{C}(1, M))=0$, 
but its $3$-adic expansion has smallest digit $1$. 
The example $M=64 =(2101)_3$ has $\dim_{H}(\mathcal{C}(1, M))> 0$,
but its $3$-adic expansion has a digit $2$.

 \begin{proof}{\em of Theorem~\ref{th31}.}
(1) Suppose the smallest nonzero $3$-adic digit in the $3$-adic expansion of the positive integer $M$ is $2$. 
Then the graph presentation of the path set $X(1, M)$
associated to $\mathcal{C}(1,M)$ constructed using Algorithm A consists of only the 
node labeled $0$ and the self-loop labeled $0$ at this node (i.e. $\mathcal{C}(1,M) = \{0\}$), 
whence  $\dim_H (\mathcal{C(}1,M)) = 0$. This holds  because 
the smallest nonzero digit of $MN$ for any $N \in \Sigma_3$ is $2$, so that $MN \notin \Sigma_3$.

(2) Suppose $M_1,\ldots , M_n \in \Sigma_3$ are positive integers so that all of their $3$-adic expansions 
have only the digits $0$ and $1$. For each $M_i$, let $m_i$ be the largest nonzero ternary position 
of $M_i$ (i.e. $M_i = 3^{m_i} +$ \emph{lower order terms}). Then in the graph presentation constructed for $X(1,M_i)$
by Algorithm A, the walk starting at the origin, then moving along an edge labeled $1$ (which exists since $(M_i)_3$ 
omits the digit $2$), then moving along $m_i$ consecutive edges labeled $0$, is a directed cycle  at $0$. Since the edge labeled $0$ is a loop at $0$,
 if we let $m = \max_{1 \leq i \leq n} m_i$, then the graph presentation the path set $X(1, M_1, ..., M_j)$
 of $\mathcal{C}(1,M_1,\ldots, M_n)$
  has a directed cycle at $0$ of length $m+1$ given by first traversing the edge labeled $1$, then traversing $m$ consecutive edges 
  labeled $0$. This cycle  and plus the loop of length one at $0$ are distinct directed cycles at $0$.
  It follows that the associated path set has 
  positive topological entropy, and hence  $\mathcal{C}(1,M_1, \ldots , M_n)$ has positive Hausdorff dimension
  by \cite[Theorem 3.1 (iii)]{AL12b}.
 \end{proof}

%
%
%
\subsection{The family $L_k= (1^k)_3 = \frac{1}{2}(3^{k}-1)$.} \label{sec42}

The path set presentations $(\mathcal{G},v_0)$ of the sets
 $\mathcal{C}(1,L_k)$ are particularly simple to analyze.

\begin{thm}\label{th32}

(1) For $k \ge 1$, and $L_k = \frac{1}{2}(3^{k}-1)$, there holds 
\begin{equation}
\dim_H(\mathcal{C}(1,L_k)) =\log_3 \beta_k,
\end{equation}
where $\beta_k$ is the unique real root greater than $1$ of
\begin{equation}
\lambda^k - \lambda^{k-1} -1 =0.
\end{equation}

(2) For $k \ge 6$, the values $\beta_k$ satisfy the bounds
\begin{equation}\label{3030}
1+ \frac{\log k}{k} - \frac{2 \log\log k}{k} \le \beta_k \le 1 +  \frac{\log k}{k}.
\end{equation}
Then for all $k \ge 3$,
\begin{equation}\label{hdL}
\dim_{H} (\mathcal{C}(1,L_k)) = \frac{\log_3 k}{k} + O \Big(\frac{\log\log k}{\log k}\Big).
\end{equation}
\end{thm} 

%
%

\begin{minipage}{\linewidth}
\begin{center}
\begin{tabular}{|c | r | r |}
\hline
 \mbox{Path set} & \mbox{Perron eigenvalue} &  \mbox{Hausdorff dim} \vline \\
\hline
 $\mathcal{C}(1,L_1)$ &   $2.000000$ & $ 0.630929$ \\
$\mathcal{C}(1,L_2)$ &   $1.618033$ & $ 0.438018$ \\
$\mathcal{C}(1,L_3)$ &  $1.465571$  &   $0.347934$  \\ 
$\mathcal{C}(1,L_4)$ &  $1.380278$  & $0.293358$  \\ 
$\mathcal{C}(1,L_5)$ &  $1.324718$  & $0.255960$   \\ 
$\mathcal{C}(1,L_6)$ &   $1.285199$    & $0.228392$ \\
$\mathcal{C}(1,L_7)$ & $1.255423$ & $0.207052$ \\
$\mathcal{C}(1,L_8)$ &  $1.232055$ & $0.189948$ \\ 
$\mathcal{C}(1,L_9)$ &  $1.213150$ & $0.175877$ \\ \hline
\end{tabular} \par
\bigskip
\hskip 0.5in {\rm TABLE 4.1.}  Hausdorff dimensions of $\mathcal{C}(1,L_k)$ (to six decimal places)
\newline
\newline
\end{center}
\end{minipage}

We first analyze the structure of the directed graph $(\mathcal{G}, v_0)$ 
in this presentation.


 \begin{prop} \label{pr42a}
   For $L_k= (1^k)_3 = \frac{1}{2}(3^k -1)$ the   path set $\mathcal{C}(1,L_k)$ has a presentation $(\sG, v_0)$  
   given by Algorithm A 
      which has exactly  $k$ vertices. The vertices  
   $v_m$ have  labels $m=0$ and $m =(1^j)_3$, for $1 \le j \le k-1$.  The underlying directed graph $G$ is strongly connected
   and primitive.
   \end{prop}

\begin{proof} The presentation $(\mathcal{G},v_0)$ of $\mathcal{C}(1,L_k)$ has an underlying directed  graph $G$
having $k$ vertices $V_n$ with $N= 0$ and $N= (1^j)_3$ for $1 \le j \le k-1$. The vertex $v_0$ has two exit edges
labeled $0$ and $1$, and all other vertices have a unique exit edge labeled $0$.
The edges form a self-loop 
at $0$ labeled $0$, and a directed $k$-cycle, whose vertex labels are 
$$
0 \to (1^{k-1})_3 \to (1^{k-2})_3 \to \cdots(1^2)_3  \to (1)_3 \to 0, 
$$
This  cycle certifies strong connectivity of the graph $G$,
and in it all edge labels are $0$ except the edge $0 \to (1^{k-1})_3$
labeled $1$. Primitivity follows because it has a cycle of length $1$ at vertex $(0)_3$.
\end{proof}

\begin{proof}[Proof of Theorem \ref{th32}]
 (1) By appropriate ordering of the vertices, the adjacency matrix $\bf{A}$ of $\mathcal{G}$ is the $k \times k$ matrix
\begin{equation*}
\bf{A} = \left(\begin{array}{ccccc}
1 & 1 & 0 &  \ldots &0\\
0 & 0 & 1 &  \ddots  & \vdots\\
\vdots & \vdots & \ddots & \ddots & 0 \\
0 & 0 & \ldots &  0 & 1 \\
1 & 0 & \ldots&  0  & 0
\end{array}\right).
\end{equation*}
The characteristic polynomial of this matrix is 
\[
p_k(\lambda) :=\det ( \lambda\bf{I} - \bf{A}) =  \det \left(\begin{array}{ccccc}
\lambda - 1 & -1 & 0 & \ldots & 0 \\
0 & \lambda & -1 &\ddots & \vdots \\
\vdots & \vdots & \ddots & \ddots & 0 \\
0 &  0 & \ldots & \lambda & -1 \\
-1 & 0 & \ldots & 0 & \lambda
\end{array}\right). 
\]
Expansion of this determinant  by minors on the first column  yields
\begin{align}
p_k(\lambda)  
= (\lambda - 1) \lambda^{k-1} + (-1)^{k-1} (-1) (-1)^{k-1}
= \lambda^k - \lambda^{k-1} - 1.
\end{align}
The Perron eigenvalue of the nonnegative matrix $\bf{A}$ will be a positive real root $\alpha_k \geq 1$ of $p(\lambda)$. 
By \eqref{top-entropy}  the topological entropy of the path set $X(1, L_k)$
associated to $C(1, L_k)$ is $\log \beta_k$,
while by Proposition \ref{pr22a} the Hausdorff dimension of the $3$-adic path set fractal $C(1,L_k)$ itself is $\log_3 \beta_k$

(2) We estimate the size of $\beta_k$.  There is at most one real root $\beta_k \ge 1$ since 
for $\lambda > 1-1/k$ one has 
\begin{eqnarray*}
p_k^{'}(\lambda) &= & k \lambda^{k-1} - (k-1)\lambda^{k-2} 
= \lambda^{k-2} (k\lambda - (k-1))  > 0.
\end{eqnarray*}

For the lower bound, we consider $p_k(\lambda)$ for $\lambda >1$
and define variables $y >0$ by $\lambda= 1+ \frac{y}{k}$ with $y>0$, and  $x := \lambda^k>1$, noting that
$
w= \lambda^k = (1+\frac{y}{k})^k < e^y
$
Now 
\[ 
\lambda^{k-1}+1 = \frac{x}{1+ \frac{y}{k}} +1 \ge x\left(1- \frac{y}{k}\right) +1 \ge x+\left(1- \frac{xy}{k}\right),
\]
which exceeds $x$ whenever $xy \le k$. Thus we have $p_k(1+ \frac{y}{k}) < 0$
whenever $xy < ye^y \le k$. The choice $y =\log k - 2 \log\log k$ gives, for $k  \ge 3$, 
\[
y e^y \le \log k (e^{\log k - 2 \log\log k}) \le \frac{k}{\log k} \le k.
\]
Thus we have, for $k \ge 3$,  $p_k( 1+ \frac{\log k}{k}- 2\frac{\log\log k}{k}) <0$, so
\[ 
\beta_k \ge 1+ \frac{\log(k)}{k}- 2\frac{\log\log k}{k},
\]
which is the lower bound in \eqref{3030}. 
For the upper bound, it suffices to  show $p_k (1+ \frac{\log k}{k}) >0$ for $k \ge 6$.
We wish to show $ (1+ \frac{\log k}{k})^{k-1} (\frac{\log  k}{k}) >1$ for $k \ge 6$. This becomes
$(1+ \frac{\log k}{k})^{k-1} > \frac{k}{\log k}$, and on taking logarithms requires
$$
(\log k -1)\log (1 + \frac{\log k}{k}) > \log k - \log\log k.
$$
Using the approximation $\log (1+x) \ge x - \frac{1}{2} x^2$ valid for $0< x< 1,$
we verify this inequality  holds for $k \ge 6$, and the upper bound in \eqref{3030} follows.
The asymptotic estimate  \eqref{hdL}  for the  Hausdorff dimension  of $\mathcal{C}(1, L_k)$
immediately follows by taking logarithms
to base $3$ of the estimates above.

\end{proof}

The results above imply Theorem \ref{th17a} in the introduction.

\begin{proof}[Proof of Theorem \ref{th17a}.]
Assertion  (1) follows from  
Proposition \ref{pr42a}.
Assertions  (2) and (3)  follow from Theorem  \ref{th32}.
\end{proof}

%
%
%

\subsection{The family $N_k= (10^{k-1}1)_3=3^k+1$.}\label{sec43}

We prove the following result.


\begin{thm}\label{th34}
 For every integer $k \geq 0$, and $N_k=3^k+1=(10^{k-1}1)_3$, 
 \begin{equation}
\dim_H(\mathcal{C}(1,N_k)) =  \dim_{H} \mathcal{C}(1, (10^{k-1}1)_3)
= \log_3\bigg(\frac{1 + \sqrt{5}}{2}\bigg) \approx 0.438018.
\end{equation}
\end{thm}

  To prove this result we first characterize the presentation $\mathcal{G}= (G, v_0)$
  associated to $N_k$ by the construction of Theorem \ref{th31n}.
  

  \begin{prop} \label{pr45}
   For $N_k= 3^k +1$ the   path set $\mathcal{C}(1,N_k)$ has a presentation $\sG =(G, v_0)$  
   given by Algorithm A  with the following properties.  
   
   (1) The vertices  
   $v_m$ have  labels $m$  that comprise those integers  $0 \le m \le \frac{1}{2}(3^k-1)$ whose $3$-adic expansion $(m)_3$
   omits the digit $2$. 
   
   (2) The directed graph $G$ has exactly $2^k$ vertices. 
   
   (3) The directed graph $G$ is strongly connected and primitive.
   \end{prop}
 
 \begin{proof}
(1) Any vertex $v_m$ reachable from $v_0$ has a $3$-adic expansion (equivalently ternary  expansion) $(m)_3$ 
 that omits the digit $2$, and has at most $k$ $3$-adic digits.
 This is proved by induction on the number of steps $n$ taken. The base case has the node $(0)_3$. For the induction step, 
 every vertex in the graph has an exit edge labeled $0$, and vertices with labels $m \equiv 0 ~(\bmod \, 3)$ also have an exit edge
 labeled $1$. The exit edges labeled $0$ map $m= (b_{k-1} b_{k-2} \cdots b_{1}b_{0})_3 $ to
  $m'= (0 b_{k-1} b_{k-2} \cdots b_{2}{b_1})_3$. The exit edges labeled $1$ map $m$ to $m' = (1 b_{k-1} b_{k-2} \cdots b_2 b_1)_3$. 
 For both types of exit edges the new vertex reached at the next step omits the digit $2$ from its $3$-adic expansion,
 completing the induction step.  

 (2) There are exactly $2^k$ possible such vertex labels $m$ in which $(m)_3$ omits the digit $2$.
 Call such vertex labels {\em admissible}. The largest such $m = \frac{1}{2}(3^k -1).$
 
 (3)  To show the graph $G_k$ is strongly connected it suffices to  establish  that:
\begin{enumerate}
\item[(R1)] Every possible such vertex l $v_m$ with admissible
label $m$ is reachable by a directed path in $G$
from the initial vertex $0 = (00\cdots 0)_3$.
\item[(R2)] All admissible vertices $v_m$ have a directed path in $G$ from $v_m$  to $v_0$.
\end{enumerate}
Note that (R1), (R2) together imply that  $G$ is strongly connected.
To show (R1), write  $m = (b_{k-1} \cdots b_0)_3$, with all $b_j =0$ or $1$, and let $i$ be the smallest index
with $b_i=1$.  Starting from $v_0$, we may add  a directed series of exit edges labeled in order
$b_i, b_{i+1}, b_{i+2}, \cdots,  b_{k-1}$ to arrive at $v_m$. Such edges exist in $G$, because  all intermediate vertices
$v_{m'}$ reached along  this path have $m' \equiv 0 ~(\bmod 3)$ so that an exit edges labeled both $0$ and  $1$
are available at that step. Indeed, the $j$-th step in the path has $(m_j)_3$ having $k-j$ initial
$3$-adic digits of $0$, and $k-1-i \leq k-1$.

To show (R2) we observe that for any vertex  $v_m$ following a path of  exit edges all labeled $0$
will eventually arrive at the vertex $v_0$. This is permissible since $(m)_3$ has all digits $0$ or $1$. 

Now $G_k$ is strongly connected, and it is primitive since it has a loop at vertex $0$.
This completes the proof. 
\end{proof}


To obtain an adjacency matrix for this graph, we must choose a suitable ordering of the vertex labels. Order the vertices of $\mathcal{G}$ recursively as follows:  the $(0^{k-1})_3$-vertex is first $I_1$, and the $(10^{k-1})_3$-vertex is second $I_2$. Now, suppose that at step $j$ we have ordered the vertices $I_1,\ldots, I_m$, in that order, with $m=2^j$. Then for $1 \le j < k$, we assert that there will be precisely $2m$ vertices, all distinct from  $I_1,\ldots,I_m$, to which some $I_i$ has an out edge. We can label these $J_{11},J_{12},\ldots,J_{m1},J_{m2}$ so that $J_{i1}$ has an in-edge labeled 0 from $I_i$, and $J_{i2}$ has an in-edge labeled 1 from $I_i$. Assuming this assertion, at the $j$-th  step we expand our ordering to $I_1 \ldots, I_m,J_{11},J_{12},\ldots,J_{m1},J_{m2}$.

\begin{prop} \label{pr47} The ordering of the vertices above is valid, and the adjacency matrix $\bf{A}$ of the underlying graph $G$ of $\mathcal{G}$ is the following $2^k \times 2^k$ matrix $\mathbf{A} = (a_{ij})$:
\begin{equation*}
a_{ij} = \left\{
\begin{array}{rl}
1 & \text{if } 1 \le i \le 2^{k-1} \text{ and } j \in \{2i-1,2i\} ; \\
1 & \text{if } 2^{k-1} < 1 \text{ and } j = 2(i -2^{k-1})-1; \\
0 & \text{otherwise}.
\end{array}\right.
\end{equation*}
This description is consistent and exhaustive, characterizing $\mathbf{A}$.
\end{prop}

To illustrate this, we have for $k =2$

\begin{equation*}
\bf{A} = \left(\begin{array}{cccc}
1 & 1 & 0 & 0 \\
0 & 0 & 1 & 1 \\
1 & 0 & 0 & 0 \\
0 & 0 & 1 & 0 \\
\end{array}\right),
\end{equation*}

while for $k = 3$ we have

\begin{equation*}
\bf{A} = \left(\begin{array}{cccccccc}
1 & 1 & 0 & 0 & 0 & 0 & 0 & 0 \\
0 & 0 & 1 & 1 & 0 & 0 & 0 & 0 \\
0 & 0 & 0 & 0 & 1 & 1 & 0 & 0 \\
0 & 0 & 0 & 0 & 0 & 0 & 1 & 1 \\
1 & 0 & 0 & 0 & 0 & 0 & 0 & 0 \\
0 & 0 & 1 & 0 & 0 & 0 & 0 & 0 \\
0 & 0 & 0 & 0 & 1 & 0 & 0 & 0 \\
0 & 0 & 0 & 0 & 0 & 0 & 1 & 0 \\
\end{array}\right).
\end{equation*}

\begin{proof} First, we address the ordering of the vertices of $\mathcal{G}$. 
According to the prescription of the proposition, $I_1 = (0)_3$, $I_2 = (10^{k-1})_3$. In the next step, 
there is an out-edge labeled $1$ from vertex $(10^{k-1})_3$ to $(110^{k-2})_3$, and an out-edge 
labeled $0$ from vertex $(10^{k-1})_3$ to vertex $(10^{k-2})_3$. 
This gives $I_3 = (110^{k-2})_3$, $I_4 = (10^{k-2})_3$. In general, for $k_1 + \cdots + k_r < k$ all nonnegative,
 if we have a vertex 
$(1^{k_1} 0^{k_2} 1^{k_3} \cdots 1^{k_r} 0^{k  -\Sigma k_i})_3$,
 it has an out-edge labeled $1$ to a vertex $(1^{k_1 + 1} 0^{k_2} 1^{k_3} \cdots 1^{k_r} 0^{k - 1-\Sigma k_i})_3$ 
 and an out-edged labeled $0$ to a vertex $(1^{k_1} 0^{k_2} 1^{k_3} \cdots 1^{k_r} 0^{k - 1-\Sigma k_i})_3$. 
 On the other hand, a vertex labeled $(1^{k_1} 0^{k_2} 1^{k_3} \cdots 1^{k_r})_3$ ending in $1$ has a single out-edge labeled $0$ to the vertex $(1^{k_1} 0^{k_2} 1^{k_3} \cdots 1^{k_r -1})_3$.

Thus, if an edge-walk originating at the $0$-vertex has label $(e_r e_{r-1} \cdots e_1 )_3$, 
the terminal vertex of this edge walk is the vertex $(e_r e_{r-1} \cdots e_1 0^{k -r})_3$. 
Now, for any vertex ending in $0$, edges labeled $0$ and $1$ are both admissible, 
which means that an edge walk labeled $e_1 e_2 \cdots e_{k}$ is admissible for all values $e_j = 0$ or 
$e_j =1$ for all $1 \leq j \leq k$. But this, then, says that all possible vertex labels from 
$\{0,1\}^{k}$ are achieved. 
Moreover, we showed above that a vertex with label from $\{0,1\}^{k}$ has out-edges only to other 
vertices labeled from $\{0,1\}^{k}$, so this is precisely the set of vertices of $\mathcal{G}$. The $r^{th}$ step of 
the vertex ordering procedure adds precisely those vertices which end in $0^{k -r}$, of which there are 
$2^{r-1} = 2 \cdot 2^{r-2}$. The procedure ends at the $k$th step with those vertices which end in $1$. 
In all, there are $2^{k}$ vertices, one for each label from $\{0,1\}^{k}$.

Now we can understand the definition of the coefficients $a_{ij}$ of the adjacency matrix $\mathbf{A}$ of the
 underlying graph $G$ of $\mathcal{G}$. Vertex $(0)_3$ maps into itself and vertex $(10^k)_3$, 
 which are ordered first and second with respect to the ordering. Thus $a_{11} = a_{12} = 1$, $a_{1j} = 0$ for $j > 2$. 
 Now suppose a vertex is ordered $i^th$ ($I_i$) at the $r^{th}$ stage, and $r \leq k-1$, 
 so that not all vertices have yet been ordered. There are $2^r$ vertices ordered so far 
 (so $1 \leq i \leq 2^r$), and the $(r+1)^{st}$ stage of the construction orders the next $2^r$
  vertices precisely so that the out-edges from vertex $I_i$ go to vertices $I_{2i-1}$ and $I_{2i}$.
   This gives the prescription for $a_{ij}$ for $1 \leq i \leq 2^{k-1}$.

Observe that the vertices $I_{2^{k-1}+1}, I_{2^{k-1} + 2}, \ldots, I_{2^k}$ have labels ending in $1$. 
Hence, such a vertex labeled $m$ has a single out-edge to the vertex labeled $(m-1)/3$. 
But if $m$ is the label of $I_{2^{k-1}+r}$, then $(m-1)/3$ is the label of $I_{2r-1}$. 
But $(2^{k-1}+r,2r-1)$ can be rewritten $(i,2(i-2^{k-1})-1)$. This gives the result.
\end{proof}

We are now ready to prove Theorem ~\ref{th34}.
\begin{proof}[Proof of Theorem ~\ref{th34}]
Let $\mathbf{A}_k$ be the adjacency matrix of the presentation of $\mathcal{C}(1,N_k)$ constructed via our algorithm. 
We directly find
a strictly positive eigenvector $\bv_k$ of $\mathbf{A}_k$
having  $\mathbf{A}_k\bv_k^T$ = $(\frac{1+\sqrt{5}}{2})\bv_k^T$. 
Here $\bv_k$  is a $ 2^k \times 1$ row vector, with $v_K^T$ its transpose,
and  let $\bv_k^{(j)} $ denote its $j$-th  entry.
The  Perron-Frobenius Theorem \cite[Theorem 4.2.3]{LM95}
then
 implies  that $\alpha = \frac{1+\sqrt{5}}{2}$ is the Perron eigenvalue of $\mathbf{A}_k$.
  Theorem ~\ref{th12}
   will then give us that 
\[\dim_H(\mathcal{C}(1,N_k)) =\log_3 \bigg(\frac{1+\sqrt{5}}{2}\bigg).\]
Let $\phi = \frac{1+ \sqrt{5}}{2}$ be the golden ratio. We define the vector $\bv_k$ recursively as follows:
\begin{enumerate}
\item[(1)] $\bv_1 = (\phi,1) = (\phi^1,\phi^0)$;
\item[(2)] If $\bv_{j-1} = (\phi^{k_1},\phi^{k_2},\ldots , \phi^{k_{2^{j-1}}})$, then
\[ 
\bv_j = (\phi^{k_1+1},\phi^{k_2+1},\ldots,\phi^{k_{2^{j-1}}+1},\phi^{k_1},\phi^{k_2},\ldots , \phi^{k_{2^{j-1}}}).
\]
\end{enumerate}
Note that $\bv_j$  is obtained from $\bv_{j-1}$ by adjoining $\phi \bv_{j-1}$ to the front of $\bv_j$.

 We need now to check that
 $\mathbf{A} \bv_k^T= \phi \bv_k^T$. We will argue by induction on $k$. The base case is easy. Now observe that if we write 
\begin{equation*}
\bf{A}_k = \left(\begin{array}{c}
T_k  \\
B_k \\
\end{array}\right)
\end{equation*}
for $T_k$ and $B_k$ each $2^{k-1} \times 2^k$ blocks, then we have
\begin{equation*}
B_{k+1} = \left(\begin{array}{cc}
B_k & 0 \\
0 & B_k \\
\end{array}\right)
\end{equation*}
and
\begin{equation*}
T_{k+1} = \left(\begin{array}{cc}
T_k & 0 \\
0 & T_k \\
\end{array}\right).
\end{equation*}
 
It follows easily from this and the definition of the vectors $\bv_k$ that if $\mathbf{A}_k \bv_k^T = \phi \bv_k^T$, then $\mathbf{A}_{k+1} \bv_{k+1}^T = \phi \bv_{k+1}^T$. This proves the theorem.

\end{proof}



\begin{proof}[Proof of Theorem \ref{th14}.]
Here (1) follows from  
Proposition \ref{pr45},
and (2) follows from Theorem  \ref{th34}.
\end{proof}

%
%
%

\subsection{Hausdorff dimension bounds for  $\mathcal{C}(1, M_1, ..., M_n)$ with $M_i$ in families }\label{sec45}

The path set structures of each of the three infinite families are compatible with
each other, as a function of $k$, so that the associated $\mathcal{C}(1, M_1, ..., M_n)$
all have positive Hausdorff dimension. We treat them separately.


\begin{thm} \label{thm-family1}
For the family $L_k= \frac{1}{2}( 3^{k} -1) = (1^k)_3$, 
for  $1 \leq k_1 < \ldots < k_n$, the pointed graph $\mathcal{G}(0,\ldots,0)$ of the path set 
$X(1, L_{k_1}, \cdots L_{k_m})$
associated to
 $\mathcal{C}(1,L_{k_1},\ldots,L_{k_n})$ 
 is isomorphic to the pointed graph $(\mathcal{G}_{k_n},0)$ presenting $\mathcal{C}(1,L_{k_n})$.
 In particular 
\begin{equation}
 \dim_H(\mathcal{C}(1,L_{k_1},\ldots,L_{k_n})) = \dim_H(\mathcal{C}(1,L_{k_n})).
\end{equation}
\end{thm}

\begin{proof} The presentation $(\mathcal{G}_k,0)$ of $\mathcal{C}(1,L_k)$ constructed with Algorithm A
 consists of a self-loop at the $0$-vertex and a cycle of length $k$ at the $0$-state. 
 Taking in Algorithm B
 the label product $\mathcal{G}_{k_1} \star \cdots \star \mathcal{G}_{k_n}$ gives a graph $\mathcal{G}$ with a self-loop at the $(0,\ldots, 0)$-vertex and a cycle
\begin{equation*}
\xymatrix{ (0,\ldots,0) \ar[r]^{1} & (1^{k_1-1}, \ldots,1^{k_n-1}) \ar[r]^0 &(1^{k_1-2},\ldots,1^{k_n-2}) \ar[r]^0 & \cdots \\ 
& \quad \cdots \ar[r]^0  & (0,\ldots, 0 , 1) \ar[r]^0 & (0,\ldots,0).
\\}
\end{equation*}
This cycle has length $k_n$. We can then see that the graph $\mathcal{G}$ is isomorphic to $\mathcal{G}_{k_n}$ by
 an isomorphism sending $(0,\ldots,0)$ to $0$.
\end{proof}

We next treat multiple intersections drawn from the second family $N_k$. 


\begin{thm}\label{th413}
For the family $N_k= 3^k+1= (10^{k-1}1)_3$ the following hold.

(1) For $1 \le k_1 < k_2 < \cdots < k_n$, one has 
\begin{equation}\label{lowerbound2}
\dim_H(\mathcal{C}(1,N_{k_1}, N_{k_2},\ldots,N_{k_n})) \geq dim_H(\mathcal{C}(1, L_{k_n+1}))
\end{equation}
Equality holds when $k_j = j$ for $1 \le j \le n$.

(2) For fixed $n \ge 1$, there holds 
\begin{equation}\label{3500}
\liminf_{k \rightarrow \infty} \dim_H(\mathcal{C}(1,N_k,\ldots,N_{k+n-1})) \geq \frac{1}{2} (\log_3 2) \approx 0.315464.
\end{equation}
In particular, $\Gamma_{\star} \ge  \frac{1}{2} (\log_3 2).$
\end{thm}  

\begin{proof}
(1) It is easy to see that the  set $\mathcal{C}(1,N_{k_1}, N_{k_2},\ldots,N_{k_n})$ contains the set
$$
Y_{k_n} : = \{ \lambda = \sum_{j=1}^{\infty} 3^{\ell_1 + \cdots + \ell_j} \in \ZZ_{3, \bar{2}}: \, \mbox{all}~~\ell_j \ge k_n +1\},
$$
(Here we allow finite sums, corresponding to some $\ell_j = +\infty$). 
This fact holds by observing that if $\lambda \in Y_{k,n}$ then $N_{k_j} \lambda \in \Sigma_{3, \bar{2}}$ for $1 \le j \le n$,
because 
$$
N_{k_j}\lambda = (\sum_{j=1}^{\infty} 3^{\ell_1 + \cdots + \ell_j}) + (\sum_{j=1}^{\infty} 3^{\ell_1 + \cdots + \ell_j+k_j}) 
$$
and the $3$-adic addition has no carry operations since all exponents are distinct.
The set $Y_{k_n}$  is a $3$-adic path set fractal 
and it is easily checked to be identical with $\mathcal{C}(1, L_{n_k+1})$, using the structure of its associated
graph. This proves \eqref{lowerbound2}. 
     To show equality holds, one must show  that allowable sequences for each of $N_{1}, N_{2}, ... ,N_{n}$
  require gaps of size at least $n+1$ between each successive nonzero $3$-adic digit in an element
  of $\mathcal{C}(1, N_1, N_2, ..., N_n).$  This can be done by induction on the current non-zero $3$-adic digit;
  we omit details.

(2) We study the symbolic dynamics of the elements of the underlying path sets in  
$\mathcal{C}(1,N_{k+j-1})$, for $1 \le j \le n$, given in
Theorem \ref{th34}, and use this to lower bound the 
Hausdorff dimension. \medskip


\noindent {\bf Claim.}  {\em The $3$-adic path set underlying $\mathcal{C}(1,N_k,\ldots,N_{k+n})$ contains all symbol sequences
which, when subdivided into successive blocks of length $2k + n$,
have every such block of the form
$$(00 \cdots 00 a_k a_{k-1} \cdots a_3 a_2 1)_3 \,\, \mbox{with each} \,\, a_i \in \{0,1\}.
$$ }

\begin{proof}[Proof of claim.] It suffices to show that all sequences split into blocks of length $2k+n$ of the form
 $(00 \cdots 00 a_k a_{k-1} \cdots a_3 a_2 1)_3$ occur in $\mathcal{C}(1,N_j)$ for each $k \leq j \leq k+n$, 
 since this will imply the statement for the label product. 
 Consider the presentation $\mathcal{G}_j$ of $\mathcal{C}(1,N_j)$ given by our algorithm. Beginning at the $0$-vertex, an edge labeled $1$ takes us to the state $(10^{j-1})_3$. From a vertex whose label ends in $0$, one may traverse an edge with label $1$ or $0$. But if we are at a vertex whose
 labeled $a0$, an edge labeled $0$ takes us to a vertex labeled $a$, 
 and an edge labeled $1$ takes us to a vertex labeled $1a$ (this is specific to the case of $N_j$). In other words, we apply the truncated shift map to our vertex label and either concatenate with $1$ on the left or not.
  It follows that from the vertex $(10^{j-1})_3$ the next $(j-1)$ edges traversed may be labeled either $0$ or $1$.

At this point the initial $1$ from $(10^{j-1})_3$ has moved to the far right of our vertex label. Therefore, our choice is restricted: 
we must traverse an edge labeled $0$. Since our vertex label, 
whatever it is, consists of only $0$'s and $1$'s, we can in any case traverse $j$ or more consecutive edges
 labeled $0$ to get back to the $0$-vertex. Thus, first traversing an edge labeled $1$, then 
 traversing edges labeled $0$ or $1$ freely for the next $(k-1)$-steps, then traversing $k+n$ edges 
 labeled $0$  and returning to the $0$-vertex, is possible in the graph $\mathcal{G}_j$ for each 
 $k \leq j \leq k+n$. It follows that all sequences of the desired form are in each $\mathcal{C}(1,N_j)$, 
 and hence in $\mathcal{C}(1,N_k \ldots, N_{k+n})$, proving the claim.
\end{proof}

With this claim in hand, we see that each block of size $(2k+n$ contains at least 
$2^{k-2}$ admissible $(2k + n)$-blocks in
 $\mathcal{C}(1,N_k,\ldots,N_{k+n})$. We conclude that the maximum eigenvalue $\beta_{n,k}$
 of the adjacency matrix of the graph $\mathcal{G}_{n,k}$ of $\mathcal{C}(1, N_k, N_{k+1}, \cdots, N_{k+n-1})$
 must satisfy  
 $(\beta_{n,k})^{2n+k} \ge 2^{k-2}.$
 This yields
 $$
  \beta_{n,k}   \ge  2^{\frac{k-2}{k+2n}}.
 $$
 and hence $\liminf_{k \to \infty} \beta_{n,k} \ge \sqrt{2}$.
 The Hausdorff dimension formula  in Proposition \ref{pr22a}  then yields
 \begin{equation}
 \limsup_{k \rightarrow \infty} \dim_{H}\big(\mathcal{C}(1,N_k,\ldots,N_{k+n})\big) \ge
  \limsup_{k \rightarrow \infty} \log_3 \beta_{n,k} \ge \frac{1}{2} \log_3 2.
  \end{equation}
 as asserted.
 
 The  lower bound  $\Gamma_{\star} \ge \frac{1}{2} \log_3 2$ follows immediately from this bound, see
 \eqref{Gamma-star}.
 \end{proof}

%
%
%

\section{Applications } \label{sec5}

We give several applications to improving bounds for the Hausdorff dimension of
various sets. 

%
%
%

\subsection{Hausdorff dimension of the generalized exceptional set $\mathcal{E}_\star(\mathbb{Z}_3)$} \label{sec51}

Theorem \ref{th413} (2) shows that there are arbitrarily large families 
$\mathcal{C}(1, N_{k_1}, ..., N_{k_n})$ having Hausdorff dimension uniformly
bounded below.
If one properly restricts the choice of the $N_{k_j}$ then one can obtain an infinite
set in this way, as was pointed out to us  by Artem Bolshakov. 
It yields a nontrivial lower bound on the Hausdorff dimension of the 
generalized exceptional set.


\begin{thm}\label{th51a} {\em (Lower Bound for Generalized Exceptional Set)}

(1) The subset $Y$ of the $3$-adic Cantor set $\Sigma_{3, \bar{2}}$ given by
$$
Y := \{ \lambda :=\sum_{j=0}^{\infty} a_j 3^j : \mbox{all}~~ a_{2k} \in \{0, 1\}, \,\, \mbox{all}\,\, a_{2k+1}=0\} \subset \mathbb{Z}_3.
$$
is a $3$-adic path set fractal having  $\dim_{H}(Y) = \frac{1}{2} \log_3 2 \approx 0.315464$. 
This set  satisfies 
$$
Y \subset \mathcal{C}(1, N_{2k+1}), \,\, \mbox{ for all} \,\,k \ge 0,
$$
where $N_k = 3^k+1$, and in  consequence
$$Y \subseteq \bigcap_{k=1}^{\infty} \mathcal{C}(1, N_{2k+1}).$$

(2) One has 
 \begin{equation}\label{bdd1}
\dim_{H} \Big( \{ \lambda \in \Sigma_{3, \bar{2}}:  \,\,  N_{2k+1} \lambda \in \Sigma_{3, \bar{2}} \,\, \mbox{for all}\, k \ge 0\}\Big)
\ge \dim_{H}(Y) = \frac{1}{2} \log_3 2.
\end{equation}
Therefore
\begin{equation} \label{bdd2}
\dim_{H} (\mathcal{E}_{\ast}) \ge \frac{1}{2} \log_3 2 = 0.315464.
\end{equation}
\end{thm}  

\begin{proof}
 (1) The $3$-adic path set fractal property of $Y \subset \Sigma_{3, \bar{2}}$ is easily established, since the
underlying graph of its symbolic dynamics is pictured in Figure ~\ref{fig51}.
The Perron eigenvalue of its adjacency matrix is $\sqrt{2}$, and its  Hausdorff dimension is
$\frac{1}{2} \log_3 2$ by Proposition \ref{pr22a}.

 \begin{figure}[ht]\label{fig51}
	\centering
	\psset{unit=1.3pt}
	\begin{pspicture}(-125,-20)(125,30)
	\newcommand{\nodeq}[2]{\node{#1}{#2}{q}}
	\nodeq{0}{-125,0}
	\nodeq{1}{125,0}
	\aline{q0}{q1}{1}
	\dline{q0}{q1}{0}{0}
	\end{pspicture}
	 \bigskip
	\newline
	\hskip 0.5in {\rm FIGURE 5.1.} Presentation of $Y$.	
	\newline
	\newline
\end{figure}

  The elements of  $Y$  can be rewritten in  the form
$\lambda =\sum_{j=0}^{\infty} b_{2j} 3^{2j},$
with all $b_{2j} \in \{0, 1\}$. We then have
$$
N_{2k+1} \lambda = \sum_{j=0}^{\infty} b_{2j}3^{2j} + \sum_{j=0}^{\infty} b_{2j} 3^{2j+2k+1} \in \Sigma_{3, \bar{2}},
$$
and the inclusion in the 
Cantor set $\Sigma_{3, \bar{2}}$ follows 
because the sets of $3$-adic exponents in the two sums on the right side are disjoint, so  there are no carry operations
in combining them under $3$-adic addition. This establishes that 
$Y \subset \mathcal{C}(1, N_{2k+1})$.

(2)  All elements $\lambda \in Y$ have $N_{2k+1} \lambda \in \Sigma_{3, \bar{2}}$ for
all $k \ge 1$. Thus  
$$
Y \subset \{ \lambda \in \Sigma_{3, \bar{2}}:  \,\,  N_{2k+1} \lambda \in \Sigma_{3, \bar{2}} \,\,\mbox{for all}\, k \ge 1\}.
$$
The result \eqref{bdd1} follows, from which \eqref{bdd2} is immediate.
\end{proof}

Theorem \ref{th110} is included as part (2) of this result.

%
%
%
\subsection{Bounds for approximations to the  exceptional set $\mathcal{E}(\mathbb{Z}_3)$}\label{sec53}

We conclude with numerical results concerning Hausdorff dimensions
of the upper approximation sets  $\mathcal{E}^{(k)}(\mathbb{Z}_3)$
to the exceptional set $\mathcal{E}(\mathbb{Z}_3)$.
Recall that the only powers of $2$ that are known to  have ternary expansions that omit the digit $2$
are $2^0= 1=(1)_3, 2^2= 4 = (11)_3$, and $2^8=256= (10111)_3$. In contrast $2^4 = 16 = (121)_3$ and $2^6=64= (2101)_3$.

We begin with empirical results about the sets $\mathcal{C}(1,2^{m_1},\ldots,2^{m_n})$ obtained via 
Algorithm A. Here we note the necessary condition $2^{2n} \equiv 1 ~(\bmod \, 3)$  for positive Hausdorff
dimension.

%
%

\begin{minipage}{\linewidth}\label{tab54}
\begin{center}
\begin{tabular}{|c |r @{.} l |}
\hline
Set & 
\multicolumn{2}{c}{Hausdorff dimension} \vline \\
\hline
$\mathcal{C}(1,2^2)$ & 0&438018 \\
$\mathcal{C}(1,2^4)$ & 0&255960 \\
$\mathcal{C}(1, 2^6)$ & 0&278002 \\
$\mathcal{C}(1, 2^8)$ & 0&287416 \\
$\mathcal{C}(1, 2^{10})$ & 0&215201 \\
$\mathcal{C}(1,2^{12})$ & 0&244002 \\
$\mathcal{C}(1,2^{14})$ & 0&267112 \\
\hline
$\mathcal{C}(1,2^2,2^4)$ & 0&\\
$\mathcal{C}(1,2^2,2^6)$ & 0& \\
$\mathcal{C}(1,2^2,2^8)$ & 0&228392 \\
$\mathcal{C}(1,2^2,2^{10})$ & 0& \\
$\mathcal{C}(1,2^4,2^6)$ & 0 &\\
$\mathcal{C}(1,2^4,2^8)$ & 0 &\\
$\mathcal{C}(1,2^4,2^{10})$ & 0& \\
$\mathcal{C}(1,2^6,2^8)$ & 0 &\\
$\mathcal{C}(1,2^6,2^{10})$ & 0& \\
$\mathcal{C}(1,2^8,2^{10})$ & 0& \\
\hline
$\mathcal{C}(1,2^2,2^8,2^{12})$ & 0& \\
$\mathcal{C}(1,2^2,2^8,2^{14})$ & 0 &\\ 
$\mathcal{C}(1,2^2,2^8,2^{16})$ & 0 & \\ \hline
\end{tabular} \par
\bigskip
\hskip 0.5in {\rm TABLE 5.2.} Hausdorff dimension of $\mathcal{C}(1,2^{m_1},\ldots,2^{m_k})$ (to six decimal places)
\newline
\newline
\end{center}
\end{minipage}

\begin{thm} \label{th15}
The following bounds hold for sets $\mathcal{E}^{(k)}(\mathbb{Z}_3)$.
\begin{eqnarray*} \dim_H(\mathcal{E}^{(2)}(\mathbb{Z}_3)) &\geq &\log_3\bigg(\frac{1 + \sqrt{5}}{2}\bigg) \approx 0.438018,\\
\dim_H(\mathcal{E}^{(3)}(\mathbb{Z}_3)) & \ge& \log_3 \beta_1 \approx  0.228392,
\end{eqnarray*}
where $\beta_1 \approx 1.28520$ is a root of $\lambda^6-\lambda^5 -1=0$.
\end{thm}

\begin{proof}
 We have
\begin{eqnarray*} 
\dim_H(\mathcal{E}^{(2)}(\mathbb{Z}_3)) &=&  \sup_{0\leq m_1 < m_2} \dim_H(\mathcal{C}(2^{m_1},2^{m_2}))\\
& \ge & \dim_H(\mathcal{C}(2^0,2^2)) = \log_3\left(\frac{1 + \sqrt{5}}{2}\right).
\end{eqnarray*}
The bound for $N_1= 2^2 = (11)_3$  follows from Theorem \ref{th14}, taking $k=1$.

We also have
\begin{eqnarray*} 
\dim_H(\mathcal{E}^{(3)}(\mathbb{Z}_3)) &=&  \sup_{0\leq m_1 < m_2<m_3} \dim_H(\mathcal{C}(2^{m_1},2^{m_2}, 2^{m_3}))\\
& \ge & \dim_H(\mathcal{C}(2^0,2^2,2^8)) 
= \log_3 \beta_1 \approx 0.228392
\end{eqnarray*}
where $\beta_1 \approx 1.28520...$ is a root of $\lambda^6- \lambda^5 -1=0.$
\end{proof}

It is unclear whether $\dim_H(\mathcal{E}^{(k)}(\mathbb{Z}_3))$ is positive for any $k \ge 4$. 
Currently  $\mathcal{C}(1,2^2, 2^8)$ is the only component of $\mathcal{E}^{(3)}(\mathbb{Z}_3)$  known  
to have positive Hausdorff dimension. At present we do not know of any set
 $\mathcal{C}(1,2^{m_1},2^{m_2},2^{m_3})$  that has  positive Hausdorff dimension.

\subsection*{Acknowledgments}  
The authors thank  Artem Bolshakov for making the key  observation
that Theorem \ref{th51a} should hold.
  W. Abram acknowledges the support of an
NSF  Graduate Research Fellowship and of the University of Michigan, where this work was carried out. 

%
%
%

\end{document}